\documentclass[a4paper,12pt]{amsart}
\usepackage{amsfonts,amsthm,amsmath,amssymb,graphicx}
\usepackage[utf8]{inputenc}
\usepackage[english]{babel}
 \usepackage[utf8]{inputenc}
\usepackage{float}
\usepackage{hyperref}
\usepackage[T1]{fontenc}
\usepackage{tikz}
\usetikzlibrary{arrows}
    \tikzset{point/.style={circle,inner sep=0pt,minimum size=3pt,fill=red}}
\usetikzlibrary{shapes.multipart}
\usetikzlibrary{matrix}
\usetikzlibrary{intersections}
\usetikzlibrary{decorations.markings}
\usetikzlibrary{positioning}
\usetikzlibrary{calc}
\usetikzlibrary{decorations.pathreplacing,shapes.misc}
\usepackage{geometry}
\usetikzlibrary{shapes.geometric}
\usetikzlibrary{decorations.pathmorphing}

\theoremstyle{plain}
\newtheorem{lem}{Lemma}[section]
\newtheorem{prop}[lem]{Proposition}
\newtheorem{thm}[lem]{Theorem}
\newtheorem{cor}[lem]{Corollary}
\newtheorem{rem}{Remark}

\theoremstyle{definition}
\newtheorem{defn}[lem]{Definition}
\newtheorem{ex}[lem]{Example}

\theoremstyle{remark}

\theoremstyle{Theorem A}
\theoremstyle{Theorem B}
\theoremstyle{Theorem C}
\theoremstyle{Theorem D}
\theoremstyle{Theorem E}

\newtheorem*{thmA}{Theorem A}
\newtheorem*{thmB}{Theorem B}
\newtheorem*{thmC}{Theorem C}
\newtheorem*{thmD}{Theorem D}

\numberwithin{equation}{section}

\DeclareMathOperator{\ext}{ext}

\DeclareMathOperator{\diam}{diam}
\DeclareMathOperator{\Eq}{Eq}

\DeclareMathOperator{\loc}{loc}
\DeclareMathOperator{\conv}{conv}

\DeclareMathOperator{\Diff}{Diff}

\newcommand{\R}{\mathbb R}

\newcommand{\Z}{\mathbb Z}
\newcommand{\N}{\mathbb N}

\renewcommand{\epsilon}{\varepsilon}
\title[Lyapunov spectrum properties ]{Lyapunov spectrum properties and continuity of the lower joint spectral radius}
\author{reza mohammadpour}

\address{Department of Mathematics, Uppsala University, Box 480, SE-75106, Uppsala, SWEDEN.}
\date{\today}

\email{reza.mohammadpour@math.uu.se}
\begin{document}

\maketitle

\begin{abstract}
In this paper we study ergodic optimization and multifractal behavior of Lyapunov exponents for matrix cocycles. We show that the restricted variational principle holds for generic cocycles (in the sense of \cite{BV1}) over mixing subshifts of finite type. We also show that the Lyapunov spectrum is equal to the closure of the set where the entropy
spectrum is positive for such cocycles. Moreover, we show the continuity of the entropy spectrum at boundary of Lyapunov spectrum  in the sense that $h_{top}(E(\alpha_{t}))\ \rightarrow h_{top}(E(\beta(\mathcal{A}))$, where $E(\alpha)=\{x\in X: \lim_{n\rightarrow \infty}\frac{1}{n}\log \|\mathcal{A}^{n}(x)\|=\alpha\}$, for such cocycles.

We prove the continuity of the lower joint spectral radius for linear cocycles under the assumption that linear cocycles satisfy a cone condition.

\end{abstract}
\tableofcontents
\section{Introduction and statement of the results}

Let $X$ be a compact metric space that is endowed with the
metric $d$. We call $(X, T)$ a \textit{topological dynamical system} (TDS), if $T: X\rightarrow X$ is a
continuous map on the compact metric space $X$.

We denote by $\mathcal{M}(X,T)$ the space of all $T$-invariant Borel probability measures on $X$. This space is a nonempty convex set and is compact with respect to the weak* topology. Moreover, we denote by $\mathcal{E}(X, T)\subset \mathcal{M}(X, T)$ the subset of ergodic measures, which are exactly the extremal points of $\mathcal{M}(X, T).$

Let $f:X\rightarrow \R$ be a continuous function. We denote by $S_{n}f(x):=\sum_{k=0}^{n-1}f(T^{k}(x))$ the \textit{Birkhoff sum}, and we call 
\begin{equation}\label{Birkhoff}
\lim_{n\rightarrow \infty} \frac{1}{n}S_{n}f(x)
\end{equation}
 a \textit{Birkhoff average}.

By Birkhoff theorem, for every $\mu \in \mathcal{M}(X, T)$ and $\mu$-almost every $x\in X$, the Birkhoff average is well-defined. We denote by  $\beta(f)$ and $\alpha(f)$ the supremum and infimum of the Birkhoff average over $x\in X$, respectively; we call these numbers the \textit{maximal and minimal ergodic averages of $f$}.

We say that $\Phi:=\{\log \phi_{n}\}_{n=1}^{\infty}$  is a \textit{subadditive potential} if each $\phi_{n}$ is a continuous positive-valued function on $X$ such that
 
 \[ 0 < \phi_{n+m}(x) \leq \phi_{n}(x) \phi_{m}(T^{n}(x)) \hspace{0,2cm} \forall x\in X, m,n \in \N.\]
 
 Furthermore,  $\Phi=\{\log\phi_{n}\}_{n=1}^{\infty}$ is said to be an \textit{almost additive potential} if there exists a constant $C \geq 1$ such that for any $m,n \in \N$, $x\in X$, we have
\[ C^{-1}\phi_{n}(x)\phi_{m}(T^{n})(x) \leq \phi_{n+m}(x)\leq C \phi_{n}(x) \phi_{m}(T^{n}(x)).\]

We also say that  $\Phi=\{\log\phi_{n}\}_{n=1}^{\infty}$ is an \textit{additive potential} if 
 \[\phi_{n+m}(x)=\phi_{n}(x)\phi_{m}(T^{n}(x)) \hspace{0,2cm} \forall x\in X, m,n \in \N;\]
in this case, $\phi_{n}(x) = e^{S_n\log \phi_1(x)}$.

If $\mu$ is an ergodic invariant probability measure, then the Birkhoff average converges to $\int f d \mu$ for $\mu$-almost all points, but there are plenty of ergodic invariant measures, for which the limit exists
but converges to a different quantity. Moreover, there are plenty of points
 which are not generic points for any ergodic measure or even for which
the Birkhoff average does not exist. So, one may ask about the size of the set of points
 \[E_{f}(\alpha)=\{x\in X : \frac{1}{n} S_{n}f(x)\rightarrow \alpha \hspace{0.2cm}\textrm{as}\hspace{0.2cm}n\rightarrow \infty \} ,\]
 which we call $\alpha$\textit{-level set of Birkhoff spectrum}, for a given value $\alpha$ from the set
 \[ L=\{\alpha \in \R: \exists x \in X \hspace{0.2cm}\textrm{and} \lim_{n\rightarrow \infty}\frac{1}{n}S_{n}f(x)=\alpha\},\] 
which we call \textit{Birkhoff spectrum.} 
 
 That size is usually
calculated in terms of topological entropy (see Subsection \ref{top_entropy}). Let $Z\subset X$, we denote by $h_{top}(T, Z)$ topological entropy of $T$ restricted to $Z$ or, simply, the topological
entropy of $Z$, denote $h_{top}(Z),$ when there is no confusion about $T$. In particular we write $h_{top}(T)$ for
$h_{top}(T,X)$. 

We investigate the end points of Birkhoff spectrum, i.e., $\alpha(f)$ and $\beta(f)$. Since $\alpha(f)=-\beta(-f)$, let us focus on the quantity $\beta$. It can also be characterized as
 \[ \beta(f)=\sup_{\mu \in \mathcal{M}(X, T)}\int f d\mu.\]
 
By the compactness of $\mathcal{M}(X, T)$, there is at least one measure $\mu \in \mathcal{M}(X, T)$ for which $\beta(f)=\int f d\mu$; such measures are called \textit{Birkhoff maximizing measures}.

It is well-known (see, e.g. \cite{Ol}, \cite{Fe03}, \cite{FFW}) when $(X, T)$ is a transitive subshift of finite type and $f$
is an additive potential, then
\[ E_{f}(\alpha)\neq \emptyset \Leftrightarrow \alpha \in \Omega :=\left\{ \int f d\mu : \mu \in \mathcal{M}(X,T)\right\},\]
 and
 \begin{align*}
 h_{top}(E_{f}(\alpha))&=\inf_{t\in \R}\{P_{f}(t)-\alpha t : t\in \R\}\\
 &=\sup\left\{h_{\mu}(T): \mu \in \mathcal{M}(X, T) \hspace{0.2cm}\textrm{with}\int f d\mu= \alpha\right\}\hspace{0.2cm}\forall \alpha\in \Omega,\stepcounter{equation}\tag{\theequation}\label{myeq1}
 \end{align*}
 where  $P_{f}(t)$ the topological pressure for a potential $tf$ (see \cite{PU}).

In the almost additive potentials case, (\ref{myeq1}) was proven by Feng and Huang \cite{FH} under certain assumptions.  In the subadditve potentials case, Feng and Huang \cite{FH} proved a similar result for $t>0$ under the upper semi continuity entropy assumption.

A natural example of subadditive potentials is matrix cocycles. More precisely, given a continuous map $\mathcal{A}:X \rightarrow GL(k, \R)$ taking values into the space $k \times k$ invertible matrices. We consider the products
 
 \[ \mathcal{A}^{n}(x)=\mathcal{A}(T^{n-1}(x)) \ldots \mathcal{A}(T(x))\mathcal{A}(x).\]
 
 The pair $(T, \mathcal{A})$ is called a \textit{linear cocycle}. That induces a skew-product dynamics $F$ on $X\times \R^{k}$ by $(x, v)\mapsto X\times \R^{k}$, whose $n$-th iterate is therefore \[(x, v)\mapsto (T^{n}(x), \mathcal{A}^{n}(x)v).\] If $T$ is invertible then so is $F$. Moreover, $F^{-n}(x)=(T^{-n}(x), \mathcal{A}^{-n}(x)v)$ for each $n\geq1$, where
\[\mathcal{A}^{-n}(x):=\mathcal{A}(T^{-n}(x))^{-1}\mathcal{A}(T^{-n+1}(x))^{-1}...\mathcal{A}(T^{-1}(x))^{-1}.\]

In general, one could consider vector bundles over $X$ instead of $X \times \R^{d}$, and then consider bundle endomorphisms that fiber over $T:X\rightarrow X.$

A simple class of linear cocycles is \textit{locally constant cocycles} which is defined as follows. Assume that $X=\{1,...,q\}^{\Z}$ is a symbolic space. Suppose that $T:X \rightarrow X$  is a shift map, i.e. $T(x_{l})_{l}=(x_{l+1})_{l}$. Given a finite set of matrices $\mathcal{A} =\{A_{1},\ldots,A_{q}\} \subset GL(k, \R)$, we define the function  $\mathcal{A}:X \rightarrow GL(k, \R)$ by $\mathcal{A}(x)=A_{x_{0}}.$ In this case, we say that$(T,\mathcal{A})$ is a locally constant cocycle.

By Kingman’s subadditive ergodic theorem, for any $\mu \in \mathcal{M}(X, T)$ and $\mu$ almost every $x\in X$ such that $\log^{+} \|\mathcal{A}\| \in L^{1}(\mu)$, the following limit, called the \textit{top Lyapunov exponent} at $x$, exists:

\begin{equation}\label{number}
\chi(x, \mathcal{A}):= \lim_{n\rightarrow \infty}\frac{1}{n}\log \|\mathcal{A}^{n}(x)\|,
\end{equation}
where $\|\mathcal{A}\|$ the Euclidean operator norm of a matrix $\mathcal{A}$, that is submultiplicative i.e.,
\[ 0< \|\mathcal{A}^{n+m}(x)\| \leq \|\mathcal{A}^{n}(x)\| \|\mathcal{A}^{m}(T^{n}(x))\| \hspace{0.2cm}\forall x\in X, m,n\in \N,\]
therefore, the potential $\Phi_{\mathcal{A}}:=\{\log \|\mathcal{A}^{n}\|\}_{n=1}^{\infty}$ is subadditive.

 Let us denote $\chi(\mu, \mathcal{A})=\int \chi(. , \mathcal{A} )d\mu.$ If the measure $\mu$ is ergodic, then $\chi(x, \mathcal{A})=\chi(\mu, \mathcal{A})$ for $\mu$-almost every $x\in X.$

Similarly to what we did for Birkhoff average (\ref{Birkhoff}), one can either maximize or minimize top Lyapunov exponent (\ref{number}); the corresponding quantities will be denoted by $\beta(\mathcal{A})$ and $\alpha(\mathcal{A})$, respectively; we call these numbers the \textit{maximal and minimal Lyapunov exponents of $\mathcal{A}$}. However, this time the maximization and the minimization problems are totally different. Even though $\beta(\mathcal{A})$ is always attained by at least one measure (which is called a \textit{Lyapunov maximizing measure}), that is not necessarily the case for the minimal Lyapunov exponent $\alpha(\mathcal{A})$. In fact, in the locally constant cocycles case, Bochi and Morris \cite{BM} investigated the continuity properties of the minimal Lyapunov exponent. They showed that $\alpha(\mathcal{A})$ is Lipschitz continuous at $\mathcal{A}$  under $1-$domination assumption. Breuillard and Sert \cite{BS} extended Bochi and Morris's result to the joint spectrum under domination condition. In this case the $\chi(\mu, \mathcal{A})$ depends continuously on the measure $\mu$.

Feng \cite{Fe03} proved (\ref{myeq1}) for continuous positive matrix-valued functions on the one side shift. He (see \cite{F}, \cite{FH}) also proved the first part (\ref{myeq1}) for locally constant cocycles under the irreducibility assumption.

 The linear cocycles generated by the derivative of a diffeomorphism map or a smooth map $T:X\rightarrow X$ on a closed Riemannian manifold $X$ and a family of maps $\mathcal{A}(x):=D_{x}T: T_{x} X\rightarrow T_{T(x)}X$ are called \textit{derivative cocycles}. Moreover, when $T:X\rightarrow X$ is an Anosov diffeomorphism (or expanding map), Bowen \cite{B} showed that there exists a symbolic coding of $T$ by a subshift of finite type. Therefore, one can replace the derivative cocycle of a uniformly hyperbolic map by a linear cocycle over a subshift of finite type. In general, we know much more about locally constant cocycles than about the more general derivative cocycles.
 
 In this paper, we are interested in linear cocycles $(T, \mathcal{A})$ generated by $GL(k, \R)-$\\valued functions $\mathcal{A}$ over two side subshifts of finite type $(\Sigma, T)$.   
 We denote by $\mathcal{L}$ the set of admissible words of $\Sigma$. For any $\mathcal{A}:\Sigma \rightarrow GL(k, \R)$ and $I\in \mathcal{L}$, we define
 \begin{equation}\label{defini1}
 \|\mathcal{A}(I)\|:=\max_{x\in [I]} \|\mathcal{A}^{|I|}(x)\|.
 \end{equation}
 
 We define a positive continuous function $\{\varphi_{\mathcal{A}, n}\}_{n\in \N}$ on $\Sigma$ such that
\[\varphi_{\mathcal{A}, n}(x):=\|\mathcal{A}^{n}(x)\|.\]
We denote by $\Phi_{\mathcal{A}}$ the subbadditive potential $\{\log \varphi_{\mathcal{A}, n}\}_{n=1}^{\infty}$.

 We say that $\mathcal{A}$ is \textit{quasi-multiplicative} if there exist $C>0$ and $m\in \N$ such that for every $I, J \in \mathcal{L}$, there exists $K\in \mathcal{L}$ with $|K|\leq m$ such that $IKJ \in \mathcal{L}$ and
\[\|\mathcal{A}(IKJ)\|\geq C \|\mathcal{A}(I)\|\|\mathcal{A}(J)\|.\]

Bonatti and Viana \cite{BV1} introduced the notion of \textit{typical
cocycles} among fiber-bunched cocycles (see Definitions \ref{typical1} and \ref{typical2} for precise formulations). We assume that $T:\Sigma \rightarrow \Sigma$ is a topologically mixing subshift of finite type. We denote by $H^{r}(\Sigma, GL(k, \R))$ the space of all $r-$H\"older continuous functions. We also denote by $H_{b}^{r}(\Sigma, GL(k, \R))$ the space of all $r-$H\"older continuous and fiber bunched functions, which says that the cocycles are nearly conformal. The set
\[\mathcal{W}:=\{\mathcal{A}\in H_{b}^{r}(\Sigma, GL(k, \R)) : \mathcal{A} \textrm{ is typical} \},\]
is open in $H_{b}^{r}(\Sigma, GL(k, \R))$, and also Bonatti and Viana \cite{BV1} proved that $\mathcal{W}$ is dense in
$H_{b}^{r}(\Sigma, GL(k, \R))$ and that its complement has infinite codimension, i.e., it is contained in finite unions of closed submanifolds with arbitrary codimension.

We denote $E(\alpha)=E_{\Phi}(\alpha)$ when there is no confusion about $\Phi.$

Our main results are Theorems A, B, C and D formulated as follows:

\begin{thmA}\label{thmA} Let $\mathcal{A}\in \mathcal{W}$. Then,
\[L=\overline{\{\alpha, \hspace{0.2cm} h_{top}(E(\alpha))>0\}}.\]
Furthermore, $\alpha \mapsto h_{top}(E(\alpha))$ is concave in $\alpha \in \mathring{L}.$
\end{thmA}

\begin{thmB}\label{thmB}
 Suppose that $\mathcal{A}:\Sigma \rightarrow GL(k, \R)$ belongs to typical functions $\mathcal{W}$. Then,
\begin{align*}
h_{top}(E(\alpha))&=
\sup\{h_{\mu}(T) : \mu\in \mathcal{M}(\Sigma, T), \chi(\mu, \mathcal{A})=\alpha\}\\
&=\inf\{P_{\Phi_{\mathcal{A}}}(q)-\alpha. q : q\in \R \} \hspace{0.2cm}\forall \alpha \in \mathring{\Omega}\stepcounter{equation}\tag{\theequation}\label{second_eq},
\end{align*}
where $\Omega:=\{\chi(\mu, \mathcal{A}) : \mu \in \mathcal{M}(\Sigma, T)\}.$
\end{thmB}

Barreira and Gelfert \cite{BG06} first obtained similar results for repellers of $C^{1+\alpha}$ maps satisfying
the cone condition (see Section \ref{5}) and bounded distortion. Feng and Huang \cite[Theorem 4.8]{FH} improved the result to subadditive potentials for $q \in \R_{+}$ (see \eqref{second_eq}), and then Park \cite[Corollary 6.6]{P} used their result for cocycles that are quasi-multiplicative. We extend the result for $q \in \R$ (see \eqref{second_eq}).

We show the continuity of the entropy spectrum of Lyapunov exponents for
generic matrix cocycles.

\begin{thmC}\label{thmC}
Suppose $\mathcal{A}_{l}, \mathcal{A}\in \mathcal{W}$ with $\mathcal{A}_{l} \rightarrow \mathcal{A}$, and $t_{l},t\in \R_{+}$ such that $t_{l}\rightarrow t.$
Let $\alpha_{t_{l}}=P_{\Phi_{\mathcal{A}_{l}}}^{'}(t_{l})$ and $\alpha_{t}=P_{\Phi_{\mathcal{A}}}^{'}(t)$. Then
\[\lim_{l\rightarrow \infty} h_{top}(E_{\mathcal{A}_{l}}(\alpha_{l})) =h_{top}(E_{\mathcal{A}}(\alpha)).\]
 Moreover,
\[ h_{top}(E(\alpha_{t})) \rightarrow h_{top}(E(\beta(\mathcal{A})) \hspace{0,2cm}\textrm{when}\hspace{0,2cm} t \rightarrow \infty.\]
\end{thmC}

We also investigate the continuity of the minimal Lyapunov exponents for general cocycles. We prove the continuity of the minimal Lyapunov exponent under a cone condition. Moreover, our result implies the continuity of the minimal Lyapunov exponent under 1-domination assumption.
\begin{thmD}\label{thmD}
Let $(X, T)$ be a topologically mixing subshift of finite type. Suppose that $\mathcal{A}_{n}, \mathcal{A}:X \rightarrow GL(k, \R)$ are matrix cocycles over $(X, T)$, and $\Phi_{\mathcal{A}}$ has bounded distortion. Assume that $(C_x)_{x\in X}$ is an invariant cone field on $X$ \footnote{See Section \ref{5}.}. Then $\alpha(\mathcal{A}_{n}) \rightarrow \alpha(\mathcal{A})$ when $\mathcal{A}_{n}\rightarrow \mathcal{A}.$
\end{thmD}

In this paper, we also obtain the high dimensional versions of Theorems A, B, and C.

The paper divides into two parts. The first part contains Sections \ref{2}, \ref{3} and \ref{4}, where we introduce necessary notions and results needed to state Theorems A, B and C as well as the proofs of Theorems A, B and C, and the second part contains Section \ref{5}, where we introduce necessary notions and results needed to state Theorem D and its proof.

 \textbf{\textit{Acknowledgements.}}
The author thanks Micha{\l} Rams for his careful reading of an earlier version of this paper and many helpful discussions. In particular, the author thanks Micha{\l} for helping with Proposition \ref{prop:add}. He would also like to thank Aaron Brown and Cagri Sert for introducing papers \cite{P} and \cite{BS}. Finally, he thanks the anonymous referee for corrections and suggestions that helped improve the paper.

Much of this work was completed at IMPAN in 2016-2020, and was partially supported by the National Science Center grant 2019/33/B/ST1/00275 (Poland).

\section{Preliminaries}\label{2}
In this section, we recall some basic facts and definitions that we need to prove main theorems.
\subsection{Symbolic dynamics}
In this section, we recall some definitions and basic facts related to subshift of finite type. For more information see \cite{LM}.

Let $Q=(q_{ij})$ be a $k\times k$ with $q_{ij}\in \{0, 1\}.$  
 The one side subshift of finite type associated to the matrix $Q$ is a left shift map $T:\Sigma_{Q}^{+}\rightarrow \Sigma_{Q}^{+}$ meaning that, $T(x_{n})_{n\in \N_{0}}=(x_{n+1})_{n\in \N_{0}}$, where $\Sigma_{Q}^{+}$ is the set of sequences 
\[ \Sigma_{Q}^{+}:=\{x=(x_{i})_{i\in \N_{0}} : x_{i}\in \{1,...,k\} \hspace{0.2cm}\textrm{and} \hspace{0.1cm}Q_{x_{i}, x_{i+1}}=1 \hspace{0.2cm}\textrm{for all}\hspace{0.1cm}i\in \N_{0}\}.\]
 Similarly, one defines two sided subshift of finite type $T:\Sigma_{Q}\rightarrow \Sigma_{Q}$, where 
\[\Sigma_{Q}:=\{x=(x_{i})_{i\in \Z} : x_{i}\in \{1,...,k\} \hspace{0.2cm}\textrm{and} \hspace{0.1cm}Q_{x_{i}, x_{i+1}}=1 \hspace{0.2cm}\textrm{for all}\hspace{0.1cm}i\in \Z\}.\]
When the matrix $Q$ has entries all equal to $1$ we say this is the \textit{full shift}. For simplicity, we denote that $\Sigma_{Q}^{+}=\Sigma^{+}$ and $\Sigma_{Q}=\Sigma$.

 We say that $i_{0}...i_{k-1}$ is an \textit{admissible word} if $Q_{i_{n},i_{n+1}}= 1$ for all $0\leq n \leq k-2$. We denote by $\mathcal{L}$ the collection of admissible words. We denote by $|I|$ the length of $I\in \mathcal{L}.$ Denote by $\mathcal{L}(n)$ the set of \textit{admissible words of length $n$}. That is, a word $i_{0},..,i_{n-1}$ with $i_{j}\in \{1,...,k\}$ such that $Q_{x_{i}, x_{i+1}}=1$.  One can define \textit{n-th level cylinder} $[I]$ as follows:
\[ [I]=[i_{0}...i_{n-1}]:=\{x\in \Sigma : x_{i}=i_{j} \hspace{0.2cm}\forall \hspace{0.1cm} 0\leq j\leq n-1 \},\]
for any $i_{0}...i_{n-1}\in \mathcal{L}(n).$

Observe that the partition of $\Sigma_Q$ (or $\Sigma_Q^+$) into first level cylinders is generating, for this reason the partition into first level cylinders is the partition canonically used in symbolic dynamics to calculate the metric entropy.

\begin{defn}
The matrix $Q$ is called primitive when there exists $n>0$ such that all the entries of $Q^n$ are positive.
\end{defn}

It is well-known that a subshift of a finite type associated with a primitive matrix $Q$ is  \textit{topologically mixing}. That is, for every open nonempty $U, V \subset \Sigma$, there is $N$ such that for every $n\geq N$, $T^{n}(U)\cap V \neq \emptyset.$ We say that $T$ is \textit{topological transitive} if there is a point with dense orbit.

We fix $\omega \in (0,1)$ and consider the space $\Sigma$ is endowed with the metric $d_{\omega}$ which is  defined as follows: For $x=(x_{i})_{i\in \Z}, y=(y_{i})_{i\in \Z} \in \Sigma$, we have 
\begin{equation}\label{metric}
d_{\omega}(x,y)= \omega^{k},
\end{equation} 
where $k$ is the largest integer such that $x_{i}=y_{i}$ for all $|i| <k.$ We denote $d:=d_{\omega}$ when there is no confusion about $\omega.$

In the two-sided dynamics, we define the \textit{local stable set}
\[ W_{\loc}^{s}(x)=\{(y_{n})_{n\in \Z} : x_{n}=y_{n} \hspace{0,2cm}\textrm{for all}\hspace{0.2cm} n\geq 0\} \]
and the \textit{local unstable set}
\[ W_{\loc}^{u}(x)=\{(y_{n})_{n\in \Z} : x_{n}=y_{n} \hspace{0,2cm}\textrm{for all}\hspace{0.2cm} n<0\} .\]

Furthermore, the global stable and unstable manifolds of $x \in \Sigma$ are
\[W^{s}(x):=\left\{y \in \Sigma: T^{n} y \in  W_{\loc}^{s}(T^{n}(x))\text { for some } n \geq 0\right\},\]
\[
W^{u}(x):=\left\{y \in \Sigma: T^{n} y \in W_{\loc}^{u}(T^{n}(x)) \text { for some } n \leq 0\right\}
.\]

If $\Sigma$ is equipped by the metric $d_{\omega}$ (see \eqref{metric}), then the two side subshift of finite type $T:\Sigma \rightarrow \Sigma$ becomes a hyperbolic homeomorphism; see \cite[Subsection 2.3]{AV10}.

\subsection{Multilinear algebra}
We recall some basic facts about the exterior algebra. We use it for studying the singular value function. 

 We denote by $\sigma_{1},...,\sigma_{k}$ the singular values of the matrix $A$, which are the square roots of the eigenvalues of the positive semi definite matrix $A^{\ast}A$ listed in decreasing order according to multiplicity.

$\{e_{1},..,e_{k}\}$ is the standard orthogonal basis of $\R^{k}$ and define
\[ \land^{l} \R^{k} :=\textrm{span} \{e_{i_{1}}\wedge e_{i_{2}} \wedge ... \wedge e_{i_{l}} : 1\leq i_{1}\leq i_{2} \leq ... \leq i_{l}\leq k\}\]
for all $l\in\{1,...,k\}$ with the convention that $\land^{0} \R^{k}=\R$. It is called the \textit{l-th exterior power of} $\R^{k}$. That is a $\binom kl$-dimensional $\R$-vector space spanned by decomposable vectors $v_1 \wedge \ldots \wedge v_k$ with the usual identifications.

We are interested in the group $k\times k$ invertible matrices of real numbers $GL(k, \R)$ that can be seen as a subset of $\R^{k^{2}}$. This space has a
topology induced from $\R^{k^{2}}$. For $A\in GL(k, \R)$, we define an invertible linear map $A^{\wedge l} : \land^{l} \R^{k} \rightarrow \land^{l} \R^{k}$ as follows
\[ (A^{\wedge l}(e_{i_{1}}\wedge e_{i_{2}} \wedge ... \wedge e_{i_{l}}))= Ae_{i_{1}}\wedge Ae_{i_{2}} \wedge ... \wedge Ae_{i_{l}}.\]

$A^{\wedge l}$ can be represented by a $\binom kl \times \binom kl$ whose entries are the $l \times l$ minors of $A$. It can be also shown that 
\[(AB)^{\wedge l}=A^{\wedge l} B^{\wedge l},\]

\[ \|A^{\wedge l}\|=\sigma_{1}(A)...\sigma_{l}(A).\]

\subsection{Fiber bunched cocycles}
We recall that $T:\Sigma \rightarrow \Sigma$ is a topologically mixing subshift of finite type. We say that $\mathcal{A}:\Sigma \rightarrow GL(k, \R)$ is a r-H\"older continuous function, if there exists $C>0$ such that
\begin{equation}\label{hol}
 \|\mathcal{A}(x)-\mathcal{A}(y)\|\leq Cd(x,y)^{r} \hspace{0,2cm} \forall x,y \in \Sigma.
\end{equation}
We fix a  $\omega \in (0,1)$ and for $r> 0$ we let $H^{r}(\Sigma, GL(k, \R))$ be the set of r-H\"older continuous functions over the shift with respect to the
metric $d_{\omega}$ on $\Sigma$.

We denote by $h_{r}(\mathcal{A})$ the smallest constant $C$ in (\ref{hol}). We equip the $H^{r}(\Sigma, GL(k, \R))$ with the distance
\[ D_{r}(A_1, A_2)=\sup_{X}\|A_1 -A_2\|+h_{r}(A_1 - A_2).\]

It is clear the locally constant functions are $\infty$-H\"older i.e, they are $r$-H\"older for every $r>0$, with bounded $h_{r}(\mathcal{A}).$

\begin{defn}\label{holonomy}
A \textit{local stable holonomy} for the linear cocycle $(T, \mathcal{A})$ is a family of matrices $H_{y \leftarrow x}^{s} \in GL(k, \R)$ defined for all $x\in \Sigma$ with $y\in W_{\loc}^{s}(x)$ such that
\begin{itemize}
\item[a)]$H_{x \leftarrow x}^{s}=Id$ and $H_{z \leftarrow y}^{s} o H_{y \leftarrow x}^{s}=H_{z \leftarrow x}^{s}$ for any $z,y \in W_{\loc}^{s}(x)$.
\item[b)] $\mathcal{A}(y)\circ H_{y \leftarrow x}^{s}=H_{T(y) \leftarrow T(x)}^{s}\circ \mathcal{A}(x).$
\item[c)] $(x, y, v)\mapsto H_{y\leftarrow x}(v)$ is continuous.
\end{itemize}
Moreover, if $y\in W_{\loc}^{u}(x)$, then similarly one defines $H_{x \leftarrow y}^{u}$ with analogous properties.
\end{defn}

According to $(b)$ in the above definition, one can extend the definition to the global stable holonomy $H_{y\leftarrow x}^{s}$ for $y\in W^{s}(x)$ not necessarily in $W_{\loc}^{s}(x)$ :
\[ H_{y\leftarrow x}^{s}=\mathcal{A}^{n}(y)^{-1} \circ H_{T^{n}(y)\leftarrow T^{n}(x)}^{s}\circ \mathcal{A}^{n}(x),\]
where 
$n\in \N$ is large enough such that $T^{n}(y)\in W_{\loc}^{s}(T^{n}(x))$. One can extend the definition the global unstable holonomy similarly.

\begin{defn}
A $r-$H\"older continuous function $\mathcal{A}$ is called \textit{fiber bunched} if for any $x\in \Sigma$, 
\begin{equation}\label{fiber}
\|\mathcal{A}(x)\|\|\mathcal{A}(x)^{-1}\|\omega^{r}<1.
\end{equation} 

\end{defn}
We say that the linear cocycle $(T, \mathcal{A})$ is fiber-bunched if its generator $\mathcal{A}$ is fiber-bunched. Recall that $H_{b}^{r}(\Sigma, GL(k, \R))$ is the family of r-H\"older-continuous and fiber bunched functions.

The geometric interpretation of the fiber bunching condition is as follows. Let $\mathcal{A}\in H_{b}^{r}(\Sigma, GL(k, \R))$. The \textit{projective} cocycle associated to $\mathcal{A}$ and $T$ is the map $\mathbb{P}F:\Sigma \times \mathbb{P}\R^{k} \rightarrow \Sigma \times \mathbb{P}\R^{k}$ given by
\[\mathbb{P}F(x, v):=(T(x), \mathbb{P} \mathcal{A}(x)v).\]

We denote by $D\mathcal{A}_{v}$ the derivative of the action $\mathbb{P}\R^{k} \rightarrow \mathbb{P}\R^{k}$ on projective space at all points $v\in \mathbb{P}\R^{k}$. Taking derivative
\[ \|D\mathcal{A}_{v}\|\leq \|\mathcal{A}\| \|\mathcal{A}^{-1}\| \hspace{0.2cm}\textrm{and}\hspace{0.2cm}  \|D\mathcal{A}_{v}^{-1}\|\leq \|\mathcal{A}\| \|\mathcal{A}^{-1}\|\]
for all $v\in \mathbb{P}\R^{k}.$ Therefore, the fiber bunching condition implies that rate of expansion (respectively, contraction) the projective cocycle $\mathbb{P}F$ at every point $x\in \Sigma$ is bounded above by $(\frac{1}\omega)^{r}$ (respectively, below by $\omega^{r}$). 

The H\"older continuity and the fiber bunched assumption on $\mathcal{A}\in H_{b}^{r}(\Sigma, GL(k, \R))$ implies the convergence of the \textit{canonical holonomy} $H^{s \diagup u}$ (see \cite{BGMV}, \cite{KS}). That means, for any $y\in W_{\loc}^{s \diagup u}(x)$,
\[H_{y\leftarrow x}^{s}:=\lim_{n\rightarrow \infty}\mathcal{A}^{n}(y)^{-1} \mathcal{A}^{n}(x) \hspace{0.2cm}\textrm{and}\hspace{0.2cm}H_{y\leftarrow x}^{u}:=\lim_{n\rightarrow -\infty}\mathcal{A}^{n}(y)^{-1} \mathcal{A}^{n}(x).\]
In addition, when the linear cocycle is fiber bunched, the canonical holonomies  vary $r-$H\"older continuously (see \cite{KS}), i.e., there exists $C>0$ such that for $y\in W_{\loc}^{s \diagup u}(x)$,
\[\|H_{x \leftarrow y}^{s \diagup u}-I\|\leq C d(x,y)^{r}.\]

In this paper, we will always work with the canonical holonomies for fiber bunched cocycles.

\begin{rem} Even though the locally constant cocycles are not necessary fiber bunched, the canonical holonomies always exist. Indeed, for every $y\in W^{s}(x)$ there exist $m$ such that $x_n=y_n$ for all $n\geq m$. Then,
\[H_{x \leftarrow y}^{s} =\mathcal{A}^{-1}(x)\cdots\mathcal{A}^{m-1}(x)^{-1} \mathcal{A}^{m-1}(y)\cdots\mathcal{A}(y).  \]

In particular
 $H_{x \leftarrow y}^{s}= Id,$ for all $y\in W_{\loc}^{s}(x)$.
Similarly, we get the existence of the unstable holonomy.
\end{rem}
\begin{rem}If a linear cocycle is not fiber bunched, then it might admit multiple holonomies (see \cite{KS16}).
\end{rem}

\subsection{Typical cocycles}
 We are going to discuss typical cocycles.  For details, one is referred to \cite{AV}, \cite{BV1} and  \cite{V}.
 
 Suppose that $p\in \Sigma$ is a periodic point of $T$, we say $p\neq z\in \Sigma$ is a \textit{homoclinic point} associated to $p$ if it is the intersection of the stable and unstable manifold of p. That is, $z\in W^{s}(p) \cap W^{u}(p)$ (see figure \ref{fig1}). The set of homoclinic points of any periodic point is dense for hyperbolic systems such as $(\Sigma, T)$.
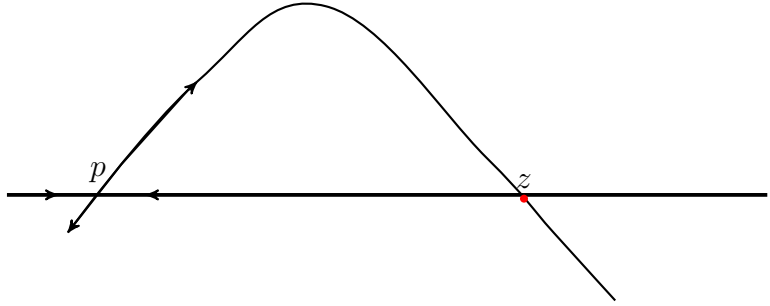
\begin{figure}[ht]
\centering
  \begin{tikzpicture}[paths/.style={<->, thick, >=stealth'}]
\draw[line width=0.5mm] (-2,3.5) -- (8,3.5)  node[pos=0.12,above]{\(p\)}  ;
\draw[>-<,thick,>=stealth'] (-1.5,3.5) -- (0,3.5);
\coordinate (G) at (0,4.3);
\coordinate (R) at (-0.9,3.1);
\coordinate (B) at (5,3);
\coordinate (D) at (8.2,5);
\draw[smooth,tension=0.7,thick] plot coordinates{(-1.2,3) (0.5,5) (2.2,6) (4.3,4) (5.19,3) (6,2.1)};
\draw[smooth,tension=0.6, <->,thick,>=stealth'] (-1.2,3) --(-0.5,3.9) --(0.39,4.9)--(0.4,4.9)--(0.5,5);
\coordinate [label=above:${z }$] (C) at (4.8,3.45);
\node[point] at (C) {};
\end{tikzpicture}
    \caption{Homoclinic point.} \label{fig1}
    \label{tikz:decision-tree}
\end{figure}

We define the \textit{holonomy loop}
\[ \psi_{p}^{z}:=H_{p \leftarrow z}^{s} \circ H_{z \leftarrow p}^{u}.\]
\begin{defn}\label{typical1}
Suppose that $\mathcal{A}:\Sigma \rightarrow GL(k, \R)$  belongs $H_{b}^{r}(\Sigma, GL(k, \R))$. We say that $\mathcal{A}$ is \textit{1-typical} if there exist a periodic point $p$ and a homoclinic point $z$ associated to $p$ such that:
\begin{itemize}
\item[(i)] the eigenvalues of  $\mathcal{A}^{per(p)}(p)$ have multiplicity $1$ and distinct absolute values. Let $\{v_{i}\}_{i=1}^{k}$ be the eigenvectors of $P:=\mathcal{A}^{per(p)}(p).$
\item[(ii)]for any $1\leq i, j\leq k$, $\psi_{p}^{z}(v_{i})$ does not lie in any hyperplane $W_{j}$, where $ W_{j}$ spanned by all eigenvectors of $P$ other than $v_{j}$.
\end{itemize}
\end{defn}
For $k=2$ this second condition means that $\psi_{p}^{z}(v_{i})\neq v_{j}$ for $1 \leq i, j \leq 2$. See Figure \ref{fig2}.\\

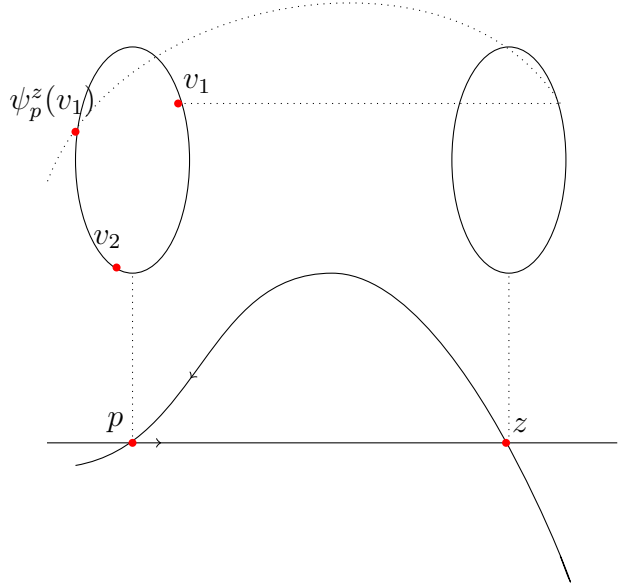
\begin{figure}[ht]
\centering
  \begin{tikzpicture}[scale=0.75]
  \tikzset{->-/.style={decoration={
  markings,
  mark=at position .2 with {\arrow{>}}},postaction={decorate}}}

\draw[dotted] (-3.5,0)--(-3.5,3)node[pos=1.2,left]{\(v_{2}\)};
\draw[dotted] (3.1,0)--(3.1,3);
\draw (-3.5,5) ellipse (1cm and 2cm);
\draw (3.1,5) ellipse (1cm and 2cm);
\draw[dotted] (-2.7,6)--(4,6) node[pos=0.05,above]{\(v_{1}\)}node[pos=-0.47,right]{\(\psi_{p}^{z}(v_{1})\)};

\draw[dotted] (4,6)  to[out=122,in=70] (-5,4.6) ;
\draw[->-] (-5,0) -- (5,0)  node[pos=0.12,above]{\(p\)} node[pos=0.83,above]{\(z\)} ;
 \tikzset{-<-/.style={decoration={
  markings,
  mark=at position .2 with {\arrow{<}}},postaction={decorate}}}
\draw[-<-] (-4.5, -0.4) to [out=10,in=-180](0,3) to[out=0,in=-70] (4,-2);
\coordinate (m1) at (-3.78,3.1);
\coordinate (Mor10) at (-2.7,6);
\coordinate (m3) at (-3.5,0);
\coordinate (m4) at (-4.5,5.5);
\coordinate (m5) at (3.05,0);
\node[point] at (m1) {};
\node[point] at (Mor10) {};
\node[point] at (m3) {};
\node[point] at (m4) {};
\node[point] at (m5) {};
\end{tikzpicture}
  \caption{$\psi_{p}^{z}(v_{1})\neq v_{2}$} \label{fig2}
  
\end{figure}
We refer to $(i)$ as the \textit{(pinching)} properties and to $(ii)$ as the \textit{(twisting)} properties.

 The cocycles generated by $\mathcal{A}^{\wedge t}$, $1\leq t\leq k$ also admit stable and unstable holonomies, namely $(H^{s \diagup u})^{\wedge t}.$
\begin{defn}\label{typical2}
Assume that $\mathcal{A}$ is 1-typical. We say $\mathcal{A}$ is \textit{t-typical} for $2\leq t \leq k-1$, if the points $p, z \in \Sigma$ from Definition \ref{typical1} satisfy
\begin{itemize}
\item[(I)] all the products of $t$ distinct eigenvalues of $P$ are distinct;
\item[(II)]  the induced map $(\psi_{p}^{z})^{\wedge t}$ on $(\R^{k})^{\wedge t}$ satisfies the analogous statement to $(ii)$ from Definition \ref{typical1} with respect to $\left\{v_{i_{1}} \wedge \ldots \wedge v_{i_{t}}\right\}_{1 \leq i_{1}<\ldots<i_{t} \leq k} \text { of } P^{\wedge t}$. 
\end{itemize}
We say that $\mathcal{A}$ is \textit{typical} if $\mathcal{A}$ is $t-$typical for all $1 \leq t \leq k-1$. We denote by $\mathcal{W} \subset H_{b}^{r}(\Sigma, GL(k, \R))$ the set of all typical functions.
\end{defn}
\begin{rem}
Above definition for typical cocycles comes from  \cite{P} that is a slightly weaker form of typical cocycles which were first introduced by Bonatti and Viana \cite{BV1}; Park \cite{P} considered a weaker twisting assumption. We also remark that
the difference between the settings of \cite{BV1} and \cite{P} does not cause
any problems in translating the relevant results and statements from \cite{BV1} to this paper. In spite of slight variations in the definition of typicality, such assumptions are
satisfied by an open and dense subset $\mathcal{W}$ of maps in $H_{b}^{r}(\Sigma, GL(k, \R))$, and its complement has infinite codimension (see \cite{AV, BV1}).
\end{rem}

\subsection{The continuity of Lyapunov exponents}
Throughout, $\mathbb{P}F: \Sigma^{+} \times \mathbb{P} \R^{k} \rightarrow \Sigma^{+}\times\mathbb{P}\R^{k}$ is the projective cocycle associated with linear cocycle $F:\Sigma^{+} \times \R^{k} \rightarrow \Sigma^{+} \times \R^{k}$ that is generated by $(T, \mathcal{A})$.

Let $(T, \mu$) be a Bernoulli shift. We say that a matrix cocycle is \textit{strongly irreducible} when there is no finite family of proper subspaces invariant by $\mathcal{A}(x)$ for $\mu$-almost every $x$. Furstenberg \cite[Theorem  6.8]{V} showed that the Lyapunov exponent $\chi(\mu, \mathcal{A})$ of $F$ coincides with the integral of the function $\psi : \Sigma^{+} \times \mathbb{P}\R^{k} \rightarrow \R,$
\[\psi(x, v)=\log \frac{\|\mathcal{A}(x)v\|}{\|v\|}\]
for locally constant cocycles under the strong irreducibility assumption. In other words, he showed that
\[ \chi(\mu, \mathcal{A})=\int \psi d(\mu \times \eta),\]
for any stationary measure $\eta$ of the associated projective cocycle $\mathbb{P}F$. So, one can easily show that we have the continuity of Lyapunov exponents with respect to $(\mathcal{A}, \mu)$ (\cite[Corollary 6.10]{V}) under the strong irreducibility assumption.

Even though discontinuity of Lyapunov exponents is a common feature (see \cite{Bo}, \cite{Boc1}), there are some results for the continuity of Lyapunov exponents. For instance, Bocker and Viana \cite{BV} proved the continuity of Lyapunov exponents of $2-$dimensional locally constant cocycles under a certain assumption. 
Avila, Eskin and Viana \cite{AEV} announced recently that Bocker and Viana's result remains true in arbitrary dimensions. It was conjectured by Viana \cite{V} that Lyapunov exponents are always continuous among $H_{b}^{\alpha}(X ,GL(2, \R))$-cocycles, and that has been proved by Backes, Brown and Butler \cite{BBB}.

We state the main result of Backes, Brown, and Butler as follows.
\begin{thm}[{{\cite[Theorem 1.1]{BBB}}}]\label{con1}
Lyapunov  exponents  vary  continuously  restricted  to  the  subset  of fiber-bunched elements $\mathcal{A}:X \to GL(2,\R).$
\end{thm}
 That improves  Bocker and Viana's result \cite{BV}. Furthermore, Butler \cite{Bu} showed in the following example that the fiber-bunching condition is sharp.

\begin{ex}\label{example}
Assume that $T:\{0,1\}^{\Z}\rightarrow \{0,1\}^{\Z}$ is a shift map. We define a locally constant cocycle $(T, \mathcal{A})$ such that
\[ A_{0}=\left[\begin{array}{ccccc}\sigma & 0\\ 0 & \sigma^{-1} 
\end{array} \right], \hspace{0.2cm}A_{1}=\left[\begin{array}{ccccc}\sigma^{-1} & 0\\ 0 & \sigma 
\end{array} \right],\]
where $\sigma$ is a positive constant greater than 1. We define probability measure $\nu_{p}$ in order to $\nu_{p}([0])=p, \nu_{p}([1])=1-p,$ and then Bernoulli measure $\mu_{p}=\nu_{p}^{\Z}$. By definition the cocycle $(T, \mathcal{A})$ is fiber bunched if and only if $\sigma^{2}<2^{\alpha}$ \footnote{$\Sigma$ is equipped by a norm $d$ that is, for all $x\neq y$, $d(x, y)=2^{-N(x,y)}$, where $N(x, y)=\min\{n, x_{n}\neq y_{n}\}$.}.
\end{ex}

Butler\cite{Bu} showed that for above example if $\sigma^{4p-2}\geq 2^{\alpha}$ for $p\in(\frac{1}{2},1)$, then for each neighborhood $\mathcal{U}\subset H^{\alpha}(\{0, 1\}^{\Z}, SL(2, \R))$ of $\mathcal{A}$ and every $\kappa \in (0, (2p-1)\log \sigma],$ there  is  a  locally  constant cocycle  $\mathcal{B}\in \mathcal{U}$ such that $\chi(x, \mathcal{B})=\kappa$. In  particular,  $\mathcal{A}$ is  a  discontinuity  point  for Lyapunov exponents in $H^{\alpha}(\{0, 1\}^{\Z}, SL(2, \R)).$ So, this example shows that we have discontinuity of  Lyapunov exponents near fiber bunched cocycles.

By considering the exterior product cocycles $\mathcal{A}^{\wedge t}$, we define
\[
\chi_{t}(\mathcal{A}(x)):=\lim _{n \rightarrow \infty} \frac{1}{n} \log \varphi_{\mathcal{A}^{\wedge t},n}(x)
\]
if the limit exists, and set
\[
\vec{\chi}(x):=\left(\chi_{1}(\mathcal{A}(x)), \ldots, \chi_{k}(\mathcal{A}(x))\right)
\]
if $\chi_{t}(\mathcal{A}(x))$ exists for each $1 \leq t \leq k$. We define the \textit{pointwise Lyapunov spectrum} of $\mathcal{A}$ as
$\vec{L}:=\left\{\vec{\alpha} \in \mathbb{R}^{k}: \vec{\alpha}=\vec{\chi}(x)\right.$ for some $\left.x \in \Sigma\right\}$. For simplicity, we denote $\chi_{t}(\mathcal{A}):=\chi_{t}(\mathcal{A}(x))$ for all $1\leq t \leq k.$
\begin{thm}[{{\cite[Theorem D]{P}}}]\label{closed}
Let $\mathcal{A}\in \mathcal{W}$. Then $\vec{L}$ is a convex and closed  subset of $\R^{k}$.
\end{thm}
We use Theorem \ref{con1} to show that the Lyapunov spectrum of fiber bunched cocycles is a closed and convex set.
\begin{cor}
Let $\mathcal{A}\in H_{b}^{r}(X, GL(2, \mathbb R)))$. Then $\vec{L}$ is a convex and closed  subset of $\R^{2}$.
\end{cor}
\begin{proof}
Since $\mathcal{W}$ is open and dense in $H_{b}^{r}(X, GL(d, \mathbb R)))$ (see \cite{BV1}), for every $\mathcal{A} \in H_{b}^{r}(X, GL(2, \mathbb R)))$ there is a $\mathcal{A}_{k} \in \mathcal{W}$ such that $\mathcal{A}_{k}\rightarrow \mathcal{A}$.

By Theorem \ref{con1},
\[ \chi_{i}(\mathcal{A}_{k})\rightarrow \chi_{i}(\mathcal{A}) \]
 By Theorem \ref{closed}, $\vec{L}$ is a closed and convex subset of $\R^{2}.$
\end{proof}

\section{Thermodynamic Formalism }\label{3}
\subsection{Convex functions}
Let $U$ be an open convex subset of $\R^{n}$ and $f$ be a real continuous convex function on $U$. We say a vector $a\in \R^{n}$ is a \textit{subgradient} of $f$ at $x$ if for all $z \in U$, 
\[f(z)\geq f(x) +a^{T}(z-x),\]
where the last term on the right-hand side is the scalar product.\\
For each $x\in \R^{n}$ set the \textit{subdifferential} of $f$ at point $x$
\[\partial f(x) :=\{a: a \hspace{0,1cm}\textrm{is a subgradient for} \hspace{0,1cm}f\hspace{0,1cm} \textrm{at}\hspace{0,1cm} x \}.\]
 For $x\in U$, the subdifferential $\partial f(x)$ is
always a nonempty convex compact set. Define $\partial^{e} f(x) := \ext\{\partial f(x)\}$, i.e., the set of exterior points $\partial f(x)$. We say that $f$ is \textit{differentiable} at $x$ if the set $\partial^{e} f(x)$ is  singleton. \\

%Notice that a subgradient can exist even when $f$ is not differentiable at $x$, as illustration consider $f(x) =|x|$. For $x <0$ the subgradient is  $\partial f(x) =\{-1\}$. Similarly, for $x >0$ we have $\partial f(x) =\{1\}$. Moreover, $\partial ^{e} f(x)=\{-1, 1\}.$ At $x= 0$ the subdifferentialis defined by the inequality $|x|\geq ax$ for every $x$, which is satisfied if and only if $a\in[-1,1]$. Therefore, we have $\partial f(0) = [-1, 1]$. This is illustrated in Figure \ref{fig2}.\\
%\begin{figure}

  %\begin{tikzpicture}

%\draw[->] (0,-2) -- (0,3)    ;
%\draw[->] (-2,0) -- (3,0)    ;
%\draw [line width=0.3mm,  red] (-2,2) -- (0,0)   node[pos=0.25,right]{\(f(x)=|x|\)} ;
%\draw [line width=0.3mm,  red] (2,2) -- (0,0)    ;

%\draw[->] (7,-2) -- (7,3)    ;
%\draw[->] (5,0) -- (10,0)    ;
%\draw [line width=0.3mm,  blue] (5,-1) -- (7,-1)   node[pos=1,right]{\(-1\)}  ;
%\draw [line width=0.3mm,  blue]  (9,1) -- (7,1)   node[pos=0,right]{\(1\)} node[pos=1.5,right]{\(\partial f\)};
%\draw [line width=0.4mm,  blue]  (7,-1) -- (7,1)    ;
%\end{tikzpicture}
    %\caption{The absolute value function (left), and its subdifferential $\partial f(x)$ as a function of $x$ (right).} \label{fig2}{%
        %
%  }
   % \label{tikz:decision-tree}
%\end{figure}
We define 
\begin{equation}\label{dense222}
\partial f(U)=\cup_{x \in U} \partial f(x) \hspace{0,2cm}\textrm{and}\hspace{0.2cm} \partial^{e} f(U)=\cup_{x \in U} \partial^{e} f(x).
\end{equation} 
%\begin{prop}{{\cite[Proposition 2.3]{FH}}}\label{useful} Let $U\subset \R^{k}$ be a non-empty open convex set. Suppose that $Z$ is a convex compact subset of a topological vector subspace which satisfies the first axiom of countability (i.e, there is a countable base at each point). Assume that $f:U\times Z\rightarrow \R \cup \{-\infty\}$ is a map satisfying following conditions:
%\begin{itemize}
%\item[1)]$f(q, z)$ is convex in $q$;
%\item[2)] $f(q, z)$ is affine in $Z$;
%\item[3)] $f$ is upper semi continuous over $U \times Z$;
%\item[4)] $g(p):=\sup_{y\in Y} f(q, z)>-\infty$ for any $q\in U$.
%\end{itemize}
%For each $q\in U$, we denote $\mathcal{I}(q)=\{z\in Z : f(q, z)=g(q)\}$. Then,
%\[\partial g(q)=\cup_{z\in \mathcal{I}(q)} \partial f(q, z),\]
%where $\partial f(q, z)$ denotes the subdifferential of $f( . , z)$ at $q$.
%\end{prop}
\subsection{Legendre transform}\label{lege} Assume that $f:\R^{k} \rightarrow \R \cup \{+\infty\}$ is a convex function that is not identically equal to $-\infty$. The Legendre transform of $f$ is the function $f^{\ast}$ of a new variable $t$, defined by
\[ t\mapsto -f^{\ast}(-t) :=\inf\{f(x)-tx : x\in \R^{k}\},\]
where $tx$ denotes the dot product of $t$ and $x$.

It is  easy to show that $f^{\ast}$ is a convex function and not identically equal to $-\infty$. Let $f^{\ast \ast}$ be the Legendre transform of $f^{\ast}$. %The following result is well known (cf. \cite[Theorem 12.2]{R}).
%\begin{thm}\label{legendre}
Assume that $f:\R^{k}\rightarrow \R \cup \{\infty\}$ is convex and not identically equal to $-\infty$. Let $x\in \R^{k}$. Suppose that $f$ is lower semi continuous at $x$, i.e., $\liminf_{y\rightarrow x} f(y)\geq f(x)$. Then $f^{\ast \ast}(x)=f(x).$
%\end{thm}
Feng and Huang \cite[Corollary  2.5]{FH} proved the following theorem:
\begin{thm}\label{Cor}
Assume that $S$ is a non-empty, convex set in $\R^{k}$ and let $g:S\rightarrow \R$ be a concave function. Set
\[W(x)=\sup\{g(a) + ax : a\in S \}, \hspace{0.3cm}x\in \R^{k}\]
and
\[G(a)=\inf\{W(x) - ax : x\in \R^{k} \}, \hspace{0.3cm}a\in S.\]
Then $G(a)=g(a)$ for $a \in \text{ri}(S)$, where $\text{ri}(S)$ denotes the relative interior of $S$ (cf. \cite{R}).
\end{thm}

\subsection{Topological entropy}\label{top_entropy}

 Assume that $(X, d)$ is a compact metric space and $T:X\rightarrow X$ is a continuous transformation. For any $n\in \N$, we define a new metric $d_{n}$ on $X$ as follows
\begin{equation}\label{new_metric}
 d_{n}(x, y)=\max\{d(T^{k}(x), T^{k}(y)) : k=0,...,n-1\},
\end{equation}
and for any $\epsilon>0$, one can define \textit{Bowen ball} $B_{n}(x, \epsilon)$ that is an open ball of radius $\epsilon>0$ in the metric $d_{n}$ around $x$. That is,
\[ B_{n}(x, \epsilon)=\{y\in X : d_{n}(x, y)<\epsilon\}.\]
 Let $Y \subset X$ and $\epsilon>0$. We say that a countable collection of balls $\mathcal{Y}:=\left\{B_{n_{i}}\left(y_{i}, \epsilon\right)\right\}_{i}$ covers $Y$ if $Y \subset \bigcup_{i} B_{n_{i}}\left(y_{i}, \epsilon\right)$. For $\mathcal{Y}=\left\{B_{n_{i}}\left(y_{i}, \epsilon\right)\right\}_{i}$, put $n(\mathcal{Y})=\min _{i} n_{i}$. Let $s\geq 0$ and define
\[ S(Y, s, N, \epsilon)=\inf \sum_{i} e^{-sn_{i}},\]
 where the  infimum  is  taken  over  all  collections  $\mathcal{Y}=\{B_{n_{i}}(x_{i}, \epsilon)\}_{i}$ covering $Y$  such  that $n(\mathcal{Y})\geq N.$  The quantity $S(Y, s, N, \epsilon)$ does not decrease with $N$, consequently 
 \[ S(Y, s, \epsilon) = \lim_{N\rightarrow \infty} S(Y, s, N, \epsilon).\]
 There is a critical value of the parameter $s$, which we denote by $h_{top}(T,Y, \epsilon)$ such that
\[S(Y, s, \epsilon)=\left\{\begin{array}{ll}
         0, & \mbox{$s>h_{top}(T,Y, \epsilon)$},\\
        \infty, & \mbox{$s<h_{top}(T,Y, \epsilon)$}.\end{array} \right . \] 
        
        Since $h_{top}(T,Y, \epsilon)$ does not decrease with $\epsilon$, the following limit exists,
        \[h_{top}(T,Y)=\lim_{\epsilon \rightarrow 0}(T,Y,\epsilon).\]
        We call $h_{top}(T, Y)$ the \textit{topological  entropy}  of $T$ restricted  to $Y$ or the topological entropy  of $Y$ (we denote $h_{top}(Y)$), as  there  is  no  confusion  about $T$. We denote $h_{top}(X, T)=h_{top}(T).$
        
\subsection{Additive thermodynamic formalism}    
    
 A \textit{potential} on $X$ is a continuous function $f: X \rightarrow \mathbb{R}$.  For any $\epsilon>0$ a set $E \subset X$ is said to be a $(n,\epsilon)$-\textit{separated  subset}  of $X$ if $d_{n}(x,y)> \epsilon$ (see \eqref{new_metric}) for any two different points $x,y \in E$. Using $(n, \varepsilon)$-separated subsets, we can define the pressure $\mathrm{P}(f)$ of $f$ as follows:
$$
P(f):=\lim _{\varepsilon \rightarrow 0} \limsup _{n \rightarrow \infty} \frac{1}{n} \log \sup \left\{\sum_{x \in E} e^{S_{n} f(x)}: E \text { is an }(n, \varepsilon) \text {-separated subset of } X\right\}.
$$
When $f \equiv 0$, the pressure $\mathrm{P}(0)$ is equal to the topological entropy $h_{top}(T)$, which measures the complexity of the system $(X, T)$.

The pressure satisfies the \textit{variational principle}:
$$
P(f)=\sup \left\{h_{\mu}(f)+\int \psi d \mu: \mu \in \mathcal{M}(X,T)\right\}
$$
where $h_{\mu}(f)$ is the measure-theoretic entropy (see \cite{PU}).
Any invariant measure $\mu \in \mathcal{M}(X,T)$ achieving the supremum in the variational principle is called an \textit{equilbrium state} of $f.$ If the entropy map $\mu \mapsto h_{\mu}(T)$ is upper semi-continuous, then every potential has an equilibrium state.   
\subsection{Subadditive thermodynamic formalism}

Let $\Phi=\{\log \phi_{n}\}_{n=1}^{\infty}$ be a subadditive potential over the TDS $(X, T)$. We introduce the $\textit{topological pressure}$ of $\Phi$ as follows. The space $X$ is endowed with a metric $d$.
  We define for a subadditive $\Phi$
\[ P_{n}(T, \Phi, \epsilon)=\sup\{\sum_{x\in E} \phi_{n}(x) : E \hspace{0,1cm}\textrm{is} \hspace{0,1cm}(n, \epsilon) \textrm{-separated subset of X} \}.\]
Since $P_{n}(T, \Phi, \epsilon)$ is a decreasing function of $\epsilon$, We define 
 \[P(T, \Phi, \epsilon)=\limsup_{n \rightarrow \infty} \frac{1}{n} \log P_{n}(T, \Phi, \epsilon),\] and
\[ P_{\Phi}(T)=\lim_{\epsilon \rightarrow 0}P(T, \Phi, \epsilon).\]
We call $P_{\Phi}(T)$ the topological pressure of $\Phi$. We denote by $P_{\Phi}(q)$ the topological pressure for a subadditive potential $q\Phi.$

Bowen \cite{B2} showed that for any H\"older continuous $\psi :X \rightarrow \R$ on a transitive hyperbolic set $(X, T)$ there exists a unique equilibrium measure $\mu$ (which is also a Gibbs state) for the additive potential $ \psi$.

Feng and K\"aenm\"aki  \cite{FK} extended the Bowen's result for subadditive potentials $t\Phi$ on a locally constant cocycle under the assumption that the matrices in $\mathcal{A}$ do not preserve a common proper subspace of $\R^{d}$(i.e. $(T, \mathcal{A})$ is irreducible).

Let $(X, T)$ be a topological dynamical system, and let $\Phi=\{\log \phi_{n}\}_{n=1}^{\infty}$ be a subadditive potential over the TDS $(X, T)$. Define
\[ \alpha(\Phi):=\liminf_{n\rightarrow \infty}\frac{1}{n} \log \inf_{x \in X}\phi_{n}(x), \hspace{0.1cm} \beta(\Phi):=\lim_{n\rightarrow \infty}\frac{1}{n} \log \sup_{x \in X}\phi_{n}(x).\]

Let $ \vec{q} = (q_{1}, ..., q_{k})\in \R_{+}^{k}$, and $\vec{\Phi}=(\Phi_{1},...,\Phi_{k})=(\{\log \phi_{n,1}\}_{n=1}^{\infty},...,\{\log \phi_{n,k}\}_{n=1}^{\infty})$. Assume that $\vec{q}.\vec{\Phi}=\sum_{i=1}^{k} q_{i}\Phi_{i}$ is a subadditive potential $\{q_{i}\log \phi_{n, i}\}_{n=1}^{\infty}.$ We can write topological pressure, maximal Lyapunov exponent, and minimal Lyapunov exponent of $\vec{\Phi}$, respectively
\[ P_{\vec{\Phi}}(\vec{q})=P(T,\vec{q}.\vec{\Phi}),\hspace{0.3cm} \vec{\beta}(\vec{\Phi})=\beta(\sum_{i=1}^{k} \Phi_{i}),\hspace{0.3cm}  \vec{\alpha}(\vec{\Phi})=\alpha(\sum_{i=1}^{k} \Phi_{i}).\]

For $\mu \in \mathcal{M}(X, T)$, we write 
\[ \mathcal{\chi}(\mu, \vec{\Phi})=(\chi(\mu, \Phi_{1}),...., \chi(\mu, \Phi_{k})),\]
where $\chi(\mu, \Phi_{i})=\lim_{n\rightarrow \infty}\frac{1}{n}\int \log \phi_{n, i}(x)d\mu(x)$ for $i=1,...,k.$ 

%\begin{thm}[{{\cite[Theorem 1.2]{FH}}}] \label{max1}
%Let $(X,T)$ be a topological dynamical systems such that $h_{top}(T)<\infty$. Assume that $\vec{\Phi}$ is a  subadditive potential on the compact metric space $X$. Then  the  pressure  function $P_{\vec{\Phi}}(\vec{t})$ is  a  continuous  real  convex  function  on $(0,\infty)$. Furthermore, $P_{\vec{\Phi}}^{'}(\infty) := \lim_{t \rightarrow \infty} \frac{P_{\vec{\Phi}}(\vec{t})}{\vec{t}}=\vec{\beta}(\vec{\Phi})$.
%\end{thm} 

 Cao, Feng and Huang \cite{CFH} proved the variational principle formula for subadditive potentials.
 \begin{thm}[{{\cite[Theorem 1.1]{CFH}}}]\label{var1} 
Let $(X,T)$ be a topological dynamical system such that $h_{top}(T)<\infty$. For $\vec{t}\in \R_{+}^{d}$, suppose that $\vec{\Phi}$ is a subadditive potential on the compact metric space $X$. Then
\[ P_{\vec{\Phi}}(\vec{t})=\sup\{h_{\mu}(T)+\vec{t}.\chi(\mu, \vec{\Phi}) 
: \mu \in \mathcal{M}(X,T) , \chi(\mu, \vec{\Phi}) \neq \infty \}.\]
\end{thm}
 Let $\vec{t}\in \R_{+}^{d}$, we denote by $\Eq(\vec{\Phi}, \vec{t})$ the collection of invariant measures $\mu$ such that 
\[ h_{\mu}(T)+\vec{t}.\chi(\mu, \vec{\Phi})= P_{\vec{\Phi}}(\vec{t}).\]

If $\Eq(\vec{\Phi}, \vec{t})\neq \emptyset$, then each element $\Eq(\vec{\Phi}, \vec{t})$ is called an $\textit{equilibrium state}$ for $\vec{t}.\vec{\Phi}$.

In the remaining part of this section, we recall some theorems about multifractal formalism for subadditive potentials.
%\begin{thm}[{{\cite[Proposition 3.2]{FH}}}]\label{convex1}
%Let $\vec{\Phi}$ be a subadditive potential over a TDS $(X, T)$ and $h_{top}(T)<\infty$. Then $P_{\vec{\Phi}}( . )$ is a real continuous convex function on $\R_{+}^k$ and \[\partial P(\R_{+}^{k})\subset (- \infty,\beta(\Phi_{1})]\times ... \times (- \infty,\beta(\Phi_{k})].\]
%\end{thm}

\begin{thm}[{{\cite[Theorem 1.1]{FH}}}]\label{level set1}
Let $(X,T)$ be a topological dynamical system such that the topological entropy $h_{top}(T)$ is finite.  Then  $E(\beta(\Phi))\neq \emptyset$ . Moreover,
\begin{align*}
 h_{top}(T,E(\beta(\Phi))) &=\sup\{h_{\mu}(T) : \mu \in \mathcal{M}(X,T), \chi(\mu, \Phi) =\beta(\Phi) \}\\
 =&\sup\{h_{\mu}(T) : \mu \in \mathcal{E}(X,T), \chi(\mu, \Phi) =\beta(\Phi) \}.\\
\end{align*}  
\end{thm}

The topological pressure is related to Lyapunov exponents in the following way.

\begin{prop}[{{\cite[Theorem 3.3]{FH}}}]\label{eq11}
Let $(X, T)$ be a topological dynamical system such that the  entropy  map $\mu\mapsto h_{\mu}(T)
$ is upper semi-continuous and $h_{top}(T)<\infty$. For $t\in \R_{+}^{k}$, suppose that $\vec{t}.\vec{\Phi}$ is a subadditive potential on the compact metric space $X$. Then,
\begin{equation}\label{subadditive}
\partial P_{\vec{\Phi}}(\vec{t}) =\{\chi(\mu_{\vec{t}}, \vec{\Phi}) : \mu\in \Eq(\vec{\Phi}, \vec{t})\}.
\end{equation} 
Moreover,  $\Eq(\vec{\Phi}, \vec{t})$ is  a  non-empty  compact  convex  subset  of $\mathcal{M}(X,T)$, for any  $t\in \R_{+}^{k}$. 
\end{prop}
 
\begin{thm}[{{\cite[Theorem 4.8]{FH}}}]\label{derivative}
Keep the assumption of Theorem (\ref{level set1}), we also assume that the entropy map $\mu \mapsto h_{\mu}(T)$ is upper semi-continuous on $\mathcal{M}(X,T)$.  If $t \in \R_{+}^{k}$ such that $\vec{t}. \vec{\Phi}$ has a unique equilibrium  state $\mu_{\vec{t}}\in \mathcal{M}(X,T)$,  then $\mu_{\vec{t}}$ is ergodic, $\nabla P_{\vec{\Phi}}(\vec{t}) = \chi(\mu_{\vec{t}}, \vec{\Phi})$, $E(\nabla P_{\vec{\Phi}}(\vec{t}))\neq \emptyset$ and $h_{top}(T,E(\nabla P_{\vec{\Phi}}(\vec{t}))) =h_{\mu_{\vec{t}}}(T).$ 
\end{thm}

\subsection{Thermodynamic formalism of linear cocycles}
Our motivation for studying of subadditive (or almost additive) thermodynamic formalism is to deal with linear cocycles. In this subsection, we recall some definitions and results that were just proved for linear cocycles.

We study \textit{ergodic optimization} of Lyapunov exponents. Ergodic optimization of Lyapunov exponents is concerned invariant measures in maximizing (or minimizing) the Lyapunov exponents. They were first considered by Rota and Strang \cite{RS} and by Gurvits \cite{GU}, respectively. The associated growth rates are called upper joint spectral radius and lower joint spectral radius, respectively; they play an important role in Control Theory (see \cite{J}, \cite{Bo}).

Let $(X,T)$ be a topological dynamical system and let $\mathcal{A}:X\to GL(k,\R)$ be a matrix cocycle over the topological dynamical system $(X,T).$

We define the \textit{maximal Lyapunov exponent} of linear cocycles as follows
\[ \beta(\mathcal{A}):=\lim_{n\rightarrow \infty}\frac{1}{n}\log \sup_{x\in X} \|\mathcal{A}^{n}(x)\|.\]

 Morris \cite{Mor10} also showed 
\begin{equation}\label{invariant1}
\beta(\mathcal{A})=\sup_{\mu \in \mathcal{M}(X,T)} \chi(\mu, \mathcal{A}).
\end{equation}

Feng and Huang \cite{FH} gave a different proof of it.

Let us define the set of \textit{maximizing measures} of $\mathcal{A}$ to be the set of measures on $X$ given by 
\[ \mathcal{M}_{\max}(\mathcal{A}):=\{\mu \in \mathcal{M}(X, T), \hspace{0,2cm} \beta(\mathcal{A})=\chi(\mu, \mathcal{A}) \}. \]
Author \cite{Moh21} showed that if a $SL(2, \mathbb{R})$ one-step cocycle satisfies pinching and twisting conditions and there exist strictly invariant cones whose images do not overlap on the Mather set then the Lyapunov-maximizing measures have zero entropy.

We also define the \textit{minimal Lyapunov exponents} of linear cocycles as follows 

\[ \alpha(\mathcal{A}):=\liminf_{n\rightarrow \infty}\frac{1}{n}\log \inf_{x\in X} \|\mathcal{A}^{n}(x)\|.\]

Similarly, the set of minimizing measures is defined as follows
\[ \mathcal{M}_{\min}(\mathcal{A}):=\{\mu \in \mathcal{M}(X, T), \hspace{0,2cm} \alpha(\mathcal{A})=\chi(\mu, \mathcal{A}) \}. \]

We remark that supremum (\ref{invariant1}) is attained, so $\mathcal{M}_{\max}$ is non-empty set. But, $\mathcal{M}_{\min}$ is not necessarily non-empty; see Section \ref{5}.

Let $\log^{+} \|\mathcal{A}\| \in L^{1}(\mu)$. We define sum of top $l$ Lyapunov exponents of measure as follows
\[ \chi_{l}(\mu, \mathcal{A}):=\lim_{n\rightarrow \infty}\frac{1}{n} \int \log \varphi_{\mathcal{A}^{\wedge l}, n}(x) d\mu(x),\]
for $\mu \in \mathcal{M}(X,T)$.

We are mainly concerned with the distribution of the Lyapunov exponents of $\mathcal{A}$.  More precisely, for any $\alpha \in \R$, define
\[ E_{\mathcal{A}}(\alpha) =\{x\in X, \chi_{1}(\mathcal{A})=\alpha \}, \]which is called the $\alpha$-level set of $\chi_{1}(\mathcal{A}).$

We also define the higher dimensional of level set of all of Lyapunov exponents as follows
\[ E_{\mathcal{A}}(\vec{\alpha}) =\{x\in X, \chi_{l}(\mathcal{A})=\alpha_{l} \}, \]
for $1\leq l \leq k.$

We denote $E(\alpha)=E_{\mathcal{A}}(\alpha)$, when there is no confusion about $\mathcal{A}.$

%Moreover, we define 
%\[\vec{\alpha}(\mathcal{A}):=(\lim \frac{1}{n} \inf_{I\in \mathcal{L}(n)} \sigma_{1}(\mathcal{A}^{n}(I)),...,\lim \frac{1}{n} \inf_{I\in \mathcal{L}(n)} \sigma_{d}(\mathcal{A}^{n}(I)).\]

We denote \[\vec{\Phi}_{\mathcal{A}}:=(\{\log \varphi_{\mathcal{A},n}\}_{n=1}^{\infty}, \{\log \varphi_{\mathcal{A}^{\wedge 2},n}\}_{n=1}^{\infty},...,\{\log \varphi_{\mathcal{A}^{\wedge k}, n}\}_{n=1}^{\infty} ).\]

 We say that $\vec{\Phi}_{\mathcal{A}}$ is \textit{(simultaneously) quasi-multiplicative} if there exist $C>0$ and $m\in \N$ such that for every $I, J \in \mathcal{L}$, there exists $K\in \mathcal{L}$ with $|K|\leq m$ such that $IKJ \in \mathcal{L}$ and
\[\|\mathcal{A}^{\wedge i}(IKJ)\|\geq C \|\mathcal{A}^{\wedge i}(I)\| \|\mathcal{A}^{\wedge i}(J)\|,\]
for $1\leq i \leq k.$

Park \cite{P} proved quasi-multiplicativity
for typical cocycles $\mathcal{W}$. The approach has
its roots in previous work of Feng \cite[Proposition 2.8]{F} who showed quasi-multiplicativity for locally constant cocycles under the irreducibility assumption.

\begin{thm}[{{\cite[Theorem F]{P}}}]\label{quasi}
Assume that $\mathcal{A}\in \mathcal{W}$. Then $\mathcal{A}$ is quasi-multiplicative. Moreover, $\vec{\Phi}_{\mathcal{A}}$ is (simultaneously) quasi-multiplicative.
\end{thm}
He \cite{P} also uses the quasi-multiplicative property $\mathcal{A}\in \mathcal{W}$ to show the continuity of the topological pressure which it states in the following theorem. 
\begin{thm}[{{\cite[Theorem B]{P}}}]\label{continuity} The map $(s, \mathcal{A})\rightarrow P_{\tilde{\Phi}_{\mathcal{A}}}(s)$ is continuous on $[0, \infty) \times \mathcal{W}$.
\end{thm}

Theorem \ref{uniq11} shows that we have the Feng and K\"aenm\"aki's result \cite[Proposition 1.2]{FK} for typical cocycles.

\begin{thm}\label{uniq11}
Let $\mathcal{A}\in \mathcal{W}$ be typical. Assume that $\vec{\Phi}_{\mathcal{A}}$ is (simultaneously) quasi-multiplicative and $\vec{t}\in \R_{+}^{k}$. Then $P_{\vec{\Phi}_{\mathcal{A}}}(\vec{t})$ has a unique equilibrium state $\mu_{\vec{t}}$ for the subadditive potential $\vec{t}.\vec{\Phi}_{\mathcal{A}}
$. Furthermore, $\mu_{\vec{t}}$ has the following Gibbs property: There exists $C \geq 1$ such that for any $n\in \N$, $[J]\in \mathcal{L}(n)$, we have
\begin{equation}\label{gibbs11}
C^{-1} \leq \frac{\mu_{\vec{t}} ([J])}{e^{-nP_{\vec{\Phi}_{\mathcal{A}}}(\vec{t})+\vec{t}.\ \vec{\Phi}_{\mathcal{A}}(x)}}\leq C,
\end{equation}
for any $x\in [J]$. Furthermore, $P_{\vec{\Phi}_{\mathcal{A}}} ( . )$ is differentiable on $\R_{+}^{k}$ and $\nabla P_{\vec{\Phi}_{\mathcal{A}}} ( \vec{t} )=\chi(\mu_{\vec{t}}, \vec{\Phi}_{\mathcal{A}}).$
\end{thm}
\begin{proof}
It is easily follows from \cite[Lemma 3.10]{P} and \cite[Proposition 3.9]{P}.
\end{proof}

\begin{thm}\label{convex11}
Assume that $h_{top}(T)<\infty $, and $\alpha(\mathcal{A})<\infty$. If $\mathcal{A} \in \mathcal{W}$, then $P_{\Phi_{\mathcal{A}}}( )$ is a real continuous convex function on $\R$. Moreover, $\alpha(\mathcal{A})$ exists \footnote{It means the limit exists.} and it is equal to $P_{\Phi_{\mathcal{A}}}^{'}(-\infty):=\lim_{t\rightarrow -\infty} \frac{P_{\Phi_{\mathcal{A}}}(t)}{t}.$ Similarly, $P_{\vec{\Phi}_{\mathcal{A}}}$ is  a real continuous convex function on $\R^{k}.$ Furthermore,
\[
\vec{\alpha}(\mathcal{A}):=\min\{\alpha_{i}, \vec{\alpha}\in \vec{L} \}=\lim_{\vec{t}\rightarrow -\infty}\frac{P_{\vec{\Phi}_{\mathcal{A}}}(\vec{t})}{\vec{t}}.
\]
\end{thm}
\begin{proof}
See \cite[Lemmas 2.2 and 2.3]{F}.  We remark that although \cite[Lemmas 2.2 and 2.3]{F} only deal with locally constant cocycles, the proof given there works for our theorem under slightly modification. Indeed, Feng uses the quasi-multiplicative properties to prove the lemmas. Since $\mathcal{A}\in \mathcal{W}$, $\vec{\Phi}_{\mathcal{A}}$ is (simultaneously) quasi-multiplicative by Theorem \ref{quasi}.
\end{proof}

Note that every almost additive potential is quasi-multiplicative, but the reverse is not true; see Example \eqref{example}. Therefore, one also can use the above theorem for almost additive potentials. 

\section{The proofs of Theorems A, B and C}\label{4}
In this section, we will first prove a more general version of Theorem B, and then Theorem B helps us deduce Theorem A. Finally, we will prove Theorem C.

We discuss multifractal formalism of typical cocycles. Our motivation is to study the multifractal formalism associated to certain iterated function systems with overlaps. For instance, the Hausdorff dimension of level sets has been calculated for 2-dimension-planar affine iterated function systems satisfying strong  irreducibility and the strong open set condition by B. B\'ar\'any, T. Jordan, A. K\"aenm\"aki, and M. Rams \cite{BJKR}. In additive potential setting, the Lyapunov exponents are equal the Birkhoff averages. In this case, the restricted varitional principle, topological entropy, and Hausdorff dimension level set has been studied by a lot of authors (see \cite{C}). We remark that Feng and Huang \cite{FH} proved (\ref{myeq1}) for almost additive potentials under certain assumptions. In fact, the
proof of Theorem \ref{addth1} is based on some ideas from that work. 
\begin{thm}\label{concave1}
Let $\mathcal{A} \in \mathcal{W}$. Suppose that $P_{\vec{\Phi}_{\mathcal{A}}}(\vec{q})\in \R$ for  each $\vec{q}\in \R^{k}$.   Then for $\vec{\alpha} \in \vec{L}$,
\begin{equation}\label{FHF}
h_{top}(T,E(\vec{\alpha})) = \inf \{P_{\vec{\Phi}_{\mathcal{A}}}(\vec{q})-\vec{\alpha}. \vec{q}: \hspace{0,2cm} \vec{q}\in \R^{k}\}.
\end{equation}
\end{thm} 
\begin{proof}
By Theorem \ref{quasi}, $\vec{\Phi_{\mathcal{A}}}$ is simultaneously quasi-multiplicative. Then, the proof follows from \cite[Theorem 4.10]{FH}.
\end{proof}
\begin{thm}\label{addth1}
Assume that $T:\Sigma \rightarrow \Sigma$ is a topologically mixing subshift of finite type on the compact metric space $\Sigma$. Suppose that $\mathcal{A}:X\rightarrow GL(k, \R)$ belongs to typical functions $\mathcal{W}$. 
Assume that $\vec{\Omega}$ is the range of the map from $\mathcal{M}(\Sigma,T)$ to $\R^{k}$ 
\[ \mu \mapsto (\chi_{1}(\mu, \mathcal{A}), \chi_{2}(\mu, \mathcal{A}),...,\chi_{k}(\mu, \mathcal{A})).\]

We define
\[ \textbf{h}(\vec{\alpha}):=\sup\{h_{\mu}(T) : \mu \in \mathcal{M}(\Sigma, T), \chi_{i}(\mu, \mathcal{A})=\alpha_{i}\},\]
where $\vec{\alpha}\in \Omega.$ Then,
\[\textbf{h}(\vec{\alpha})=\inf\{P_{\vec{\Phi}_{\mathcal{A}}}(\vec{q})-\vec{\alpha}. \vec{q} : \vec{q}\in \R^{k} \},\]
for $\vec{\alpha} \in \text{ri}(\vec{\Omega}).$
\end{thm}
\begin{proof}
It is clear that $\vec{\Omega}$ is non-empty and convex. We define $g:\vec{\Omega} \to \R$ by
\[g(\vec{\alpha})=\sup\{h_{\mu}(T): \mu \in \mathcal{M}(\Sigma, T), \quad (\chi_{1}(\mu, \mathcal{A}), \ldots, \chi_{k}(\mu, \mathcal{A}))=\vec{\alpha}\}.\]

It is easy to see that $g$ is a real-valued concave function on $\vec{\Omega}.$ We define
\[Z(\vec{x}):=\sup\{g(\vec{\alpha})+\vec{\alpha}.\vec{x}: \vec{\alpha}\in \vec{\Omega}\}, \quad \forall \vec{x} \in \R^{k}.\]

Let $f:\R^{k} \to \R \cup \{+\infty\}$ be the function which agrees with $-g$ on $\vec{\Omega}$ but is $+\infty$ everywhere else. Then, $f$ is convex  and has $\vec{\Omega}$ as its effective domain, i.e. $\vec{\Omega}=\{\vec{x}, f(\vec{x})< \infty \}.$ By the definition of $Z$, $Z$ is equal to the conjugate function of $f$, so $Z$ is a convex function on $\R^{k}$ (see Subsection \ref{lege}).

We have
\[g(\vec{\alpha})=\inf\{Z(\vec{x})-\vec{\alpha}.\vec{x}: \vec{x}\in \R^{k}\},\] 
for all $\vec{\alpha}\in \text{ri}(\vec{\Omega})$, by Theorem \ref{Cor}.

By Theorem \ref{quasi}, $\vec{\Phi_{\mathcal{A}}}$ is simultaneously quasi-multiplicative. Therefore,  $P_{\vec{\Phi}_{\mathcal{A}}}(\vec{q})$ is a convex function on $\R^{k}$ by Theorem \ref{convex11}. Then, by variational principle 
\[Z(\vec{x})=P_{\vec{\Phi}_{\mathcal{A}}}(\vec{x}) .\]
\end{proof}

\begin{proof}[Proof of Theorem B] Theorem B is just the one-dimensional version
of Theorems \ref{concave1} and \ref{addth1}. Also, note that $\mathring{\Omega}=\text{ri}(\Omega).$
\end{proof}
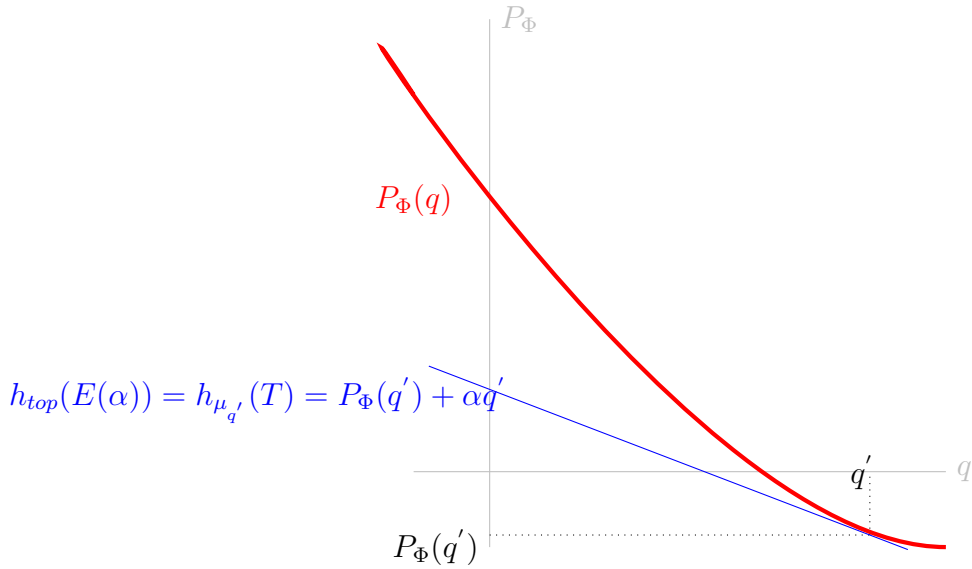
\begin{figure}[ht]
\begin{tikzpicture}
\tikzset{
    show curve controls/.style={
        decoration={
            show path construction,
            curveto code={
                \draw [blue, dashed]
                    (\tikzinputsegmentfirst) -- (\tikzinputsegmentsupporta)
                    node [at end, cross out, draw, solid, red, inner sep=2pt]{};
                \draw [blue, dashed]
                    (\tikzinputsegmentsupportb) -- (\tikzinputsegmentlast)
                    node [at start, cross out, draw, solid, red, inner sep=2pt]{};
            }
        }, decorate
    }
}

\draw [gray!50]  (-1,0) -- (6,0)node[black][pos=0.8,right]{\(q^{'}\)} node[pos=1.0,right]{\(q\)} ;
\draw [gray!50]  (0,-1) -- (0,6)node[black][pos=0.0,left]{\(P_{\Phi}(q^{'})\)}node[pos=1.0,right]{\(P_{\Phi}\)} ;
\draw[dotted]    (0, -0.84)-- (5, -0.84);
\draw[dotted]    (5, -0.84)-- (5, 0);
\draw [red] plot [smooth cycle] coordinates {(-1,4) (6,-1)} ;

\draw[blue] (5.5, -1.0331)--(5, -0.84)-- (-0.8,1.4) node[pos=0.8,left]{\(h_{top}(E(\alpha))=h_{\mu_{q^{'}}}(T)=P_{\Phi}(q^{'})+\alpha q^{'}\)};
 \draw [ultra thick,red] (-1,5) to[out=123,in=180](6,-1);

\node[red]at (-1,3.6){\(P_{\Phi}(q)\)} ;

\end{tikzpicture}
\caption{$P_{\Phi}(. )$ is a convex function for $q\in \R$. The blue line is tangent to  $P_{\Phi}(.)$ at $q$ with slope $-\alpha=P_{\Phi}'(q)$.} \label{fig3}
\end{figure}

\begin{rem}In the locally constant cocycles case with Bernoulli measures, Theorem \ref{addth1} is true for $\vec{\alpha}\in \vec{\Omega}$ under strongly irreducible assumption, which means we do not need pinching assumption in this case.
\end{rem}

%\begin{equation}\label{myself}
%\sup\{h_{\mu}(T) : \mu \in \mathcal{M}(X , T), \chi(\mu, \mathcal{A})=\vec{\alpha}\}=\inf \{P_{\tilde{\phi}_{\mathcal{A}}}(\vec{q})-\vec{\alpha}. \vec{q}: \hspace{0,2cm} \vec{q}\in \R^{d}\}.
%\end{equation}

%\begin{proof}

%See \cite[Theorem 3.3]{FH} for first part of Theorem.\\
%We are going to show that second part for $\mathcal{A}\in \mathcal{W}$.\\
%According to Theorem \ref{convex11} $P( . )$ is convex for $q\in \R^{d}$. Take $Z=\mathcal{M}(X, T)$ and $U=\R^{d}$ in Proposition \ref{useful}. One can define the following map 
%\[f(q, \mu)=q. \chi(\mu, \mathcal{A})+h_{\mu}(T),\]
%for $q\in \R^{d}$ and $\mu \in \mathcal{M}(X, T).$ Then $f$ satisfies the all of conditions in Proposition \ref{useful}. 
%\end{proof}
Now, we are going to show that the closure of the set where the entropy spectrum is positive is equal the Lyapunov spectrum for typical cocycles. This result is first attempt to characterize Lyapunov spectrum as a set of positive entropy spectrum. The main input of our argument will be the fact that the topological pressure is convex for typical cocycles, and Theorems \ref{concave1} and \ref{addth1}. Then, we can show the concavity of the entropy spectrum of Lyapunov exponents by Theorem B.

We recall that $T:\Sigma\rightarrow \Sigma$ is a topologically mixing subshift of finite type and $\mathcal{A}:\Sigma \rightarrow GL(k, \R)$ is a H\"older continuous function. We always assume $h_{\rm top}(T)>0$.

\begin{lem}\label{concave2}
Let $\mathcal{A} \in \mathcal{W}$. Then, $h_{\rm top}(E(\alpha))$ is concave on the convex set $\mathring{L}$.
\end{lem}
\begin{proof}
The topological pressure $P_{\Phi_{\mathcal{A}}}( .)$ is convex by Theorem \ref{convex11} and $\mathring{L}$ is  convex by Theorem \ref{closed}. Moreover, by Theorem \ref{concave1}, we have
\[
 h_{\rm top}(E(\alpha))=\inf_{t\in \R}\{P_{\Phi_{\mathcal{A}}}(t)- \alpha t\}.\]
Since the Legendre transform of the convex function is concave (cf. \cite[Theorem 1.1.2]{HL}), $h_{\rm top}(E(\alpha))$ is concave.
\end{proof}

\begin{lem}\label{positive-new}
Assume that a nonnegative function $f$ defined on a convex domain $D$ is concave and achieves a positive value at some point $x\in D$. Then $f(y)>0$ for all $y\in \mathring{D}$.
\end{lem}
 \begin{proof}
Let $x$ be in $D$ such that $f(x)>0$. For any point $y$ in the interior of $D$, we can always choose a point $z\in \mathring{D}$ such that:
 \[ (1-\lambda)x+\lambda z=y,\]
for some $\lambda \in (0,1)$. Hence,
 \[(1-\lambda)f(x)+\lambda f(z) \leq  f(y).\]
 Therefore, $f(y)>0$ for all $y\in \mathring{D}$.
 \end{proof}

\begin{thm}\label{posi1}
For $\alpha \in \mathring{L}$, $h_{\rm top}(E(\alpha))>0.$
\end{thm}

\begin{proof}
By Lemma \ref{concave2}, $h_{\rm top}(E(\alpha))$ is concave. Moreover, by Theorems \ref{level set1}, \ref{concave1} and \ref{addth1},
\[h_{\rm top}(E(\alpha))=\sup\{h_{\mu}(T) :\hspace{0.2cm}  \mu\in \mathcal{M}(\Sigma, T),\hspace{0.2cm} \chi(\mu, \mathcal{A})=\alpha\}.\]
Since the measure-theoretic entropy is non-negative, $h_{\rm top}(E(\alpha))\geq0$.

We claim that there is $\alpha$ such that $h_{\rm top}(E(\alpha))>0.$ Let us assume $h_{\rm top}(E(\alpha))=0$ for all $\alpha \in \mathring{L}$. Then, as by Oseledets Theorem for every ergodic measure $\mu$ supported on $(\Sigma, T)$ there exist a common value of Lyapunov exponent shared $\mu$-almost everywhere, we must have $h(\mu)=0$ for every ergodic measure $\mu$. Thus, by variational principle, $h_{\rm top}(T) = \sup_{\mu} h(\mu) = 0$, which is a contradiction. Consequently, by Lemma \ref{positive-new}, $h_{\rm top}(E(\alpha))>0$ for all $\alpha \in \mathring{L}.$
\end{proof}
\begin{rem} Entropy spectrum at boundary of Lyapunov spectrum is not necessarily positive. In fact, there is a conjecture, which is known as \textit{Meta conjecture},  that says that under generic assumptions the entropy spectrum at boundary of Lyapunov spectrum is zero (which would mean that $h_{top}(E(\beta(\mathcal{A})))=0$);  this phenomenon is often referred to as \textit{ergodic optimization of Lyapunov exponents}, see for example \cite{Bo}. In the additive potential case, instead, this phenomenon is often referred to as \textit{ergodic optimization of Birkhoff averages}, see for example \cite{J}.
\end{rem}

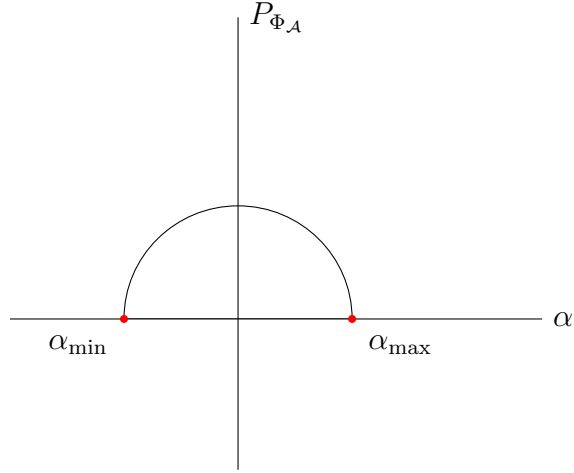
\begin{figure}[ht]

\centering

\begin{tikzpicture}[baseline=(current bounding box.north)]

% A clipped circle is drawn
\begin{scope}
    \clip (-1.5,0) rectangle (1.5,1.5);
    \draw (0,0) circle(1.5);
    \draw (-1.5,0) -- (1.5,0);
\end{scope}
%
%%Labels for the vertices are typeset.
\node[below left= 1mm of {(-1.5,0)}] {$\alpha_{\min}$};
\node[below right= 1mm of {(1.5,0)}] {$\alpha_{\max}$};
\node[above right= 1mm of {(60:1.5)}] {};
\node[above left= 1mm of {(120:1.5)}] {};

\draw[black]    (0, -2)-- (0, 4)node[pos=1.0,right]{\(P_{\Phi_{\mathcal{A}}}\)};
\draw[black]    (-3, 0)-- (4, 0)node[pos=1.0,right]{\(\alpha\)};
\node[point] at (-1.5, 0) {};
\node[point] at (1.5, 0) {};
\end{tikzpicture}

\end{figure}

\begin{proof}[Proof of Theorem A]
That is a direct consequence of Theorem \ref{posi1}.
\end{proof}

 Park \cite{P} proved Theorem \ref{closed} for higher dimensional case. That means, $\vec{L}$ is closed and convex. So, we can obtain the following generalization of Theorem A to the Lyapunov spectrum of of all Lyapunov exponents. 
\begin{thm}\label{dense}
$\overline{\{ \vec{\alpha} \in \R^{k}, \hspace{0,2cm} h_{top}(E(\vec{\alpha}))>0 \}}= \vec{L}.$
\end{thm}
We remark that
the concavity of a function defined on a convex set implies the continuity of
the function in the interior, and that
the continuity of the entropy under the change of the Lyapunov exponents
implies the continuity of the Lyapunov spectrum.

We will discuss the continuity of the entropy spectrum of Lyapunov exponents, that is, the topological entropy of level sets of points with a common given Lyapunov exponent.  In the locally constant cocycles case, Lemma \ref{convergence1} follows from \cite{FS}.

\begin{lem}\label{convergence1}
Assume  $\mathcal{A}_{k}, \mathcal{A}\in \mathcal{W}$ with $\mathcal{A}_{k} \rightarrow \mathcal{A}$. For $t_{k}, t>0$, let $t_{k}\rightarrow t$.
Suppose that $\alpha_{t_{k}}$ and $\alpha_{t}$ are the derivatives of $P_{\Phi_{\mathcal{A}_{k}}}( )$ and $P_{\Phi_{\mathcal{A}}}( )$ at $t_{k}$ and $t$, respectively. Then,
\[\lim_{k\rightarrow \infty} h_{top}(E_{\mathcal{A}_{k}}(\alpha_{t_{k}})) =h_{top}(E_{\mathcal{A}}(\alpha_{t})).\]
\end{lem}
\begin{proof}
According to Theorem \ref{uniq11}, $P_{\Phi_{\mathcal{A}}}( )$ is differentiable for any $t>0$ and there is a unique equilibrium measure $\mu_{t}$ for the subadditive potential $t\Phi_{\mathcal{A}}$ . Therefore, we have \[h_{top}(E_{\mathcal{A}}(\alpha_{t}))=h_{\mu_{t}}(T),\]  where $P_{\Phi_{\mathcal{A}}}^{'}(t)=\alpha_{t}$, by Theorem \ref{derivative}.

Taking into account the observation above, to prove the theorem it is enough to show that $h_{\mu_{t_{k}}}(T) \rightarrow h_{\mu_{t}}(T)$.

By the definition of $\Eq(\Phi_{\mathcal{A}_{k}}, t_{k})$,
\[P_{\Phi_{\mathcal{A}_{k}}}(t_{k})=h_{\mu_{t_{k}}}(T)+t_{k}\chi(\mu_{t_{k}},\mathcal{A}_{k}).\]

Notice that the Lyapunov exponents are upper semi-continuous. Moreover, the topologically mixing subshift of finite type $T:\Sigma\rightarrow \Sigma$ implies upper semi-continuity of the entropy map $\mu \mapsto h_{\mu}(T)$. Now, we conclude from above observations and Theorem \ref{continuity},
\begin{align*}
P_{\Phi_{\mathcal{A}}}(t)&=\lim_{k\rightarrow \infty} P_{\Phi_{\mathcal{A}_{k}}}(t_{k})\\&=\lim_{k\rightarrow \infty} h_{\mu_{t_{k}}}(T)+t_{k}\chi(\mu_{t_{k}},\mathcal{A}_{k})\\
&\leq h_{\mu_{t}}(T)+t\chi(\mu_{t},\mathcal{A}).
\end{align*}
This shows $\mu_{t} \in \Eq(\Phi_{\mathcal{A}}, t)$. By weak* compactness $\mu_{t_{k}}$ has a accumulation point, let us call $\mu_{t}$. According the above observation, $\mu_{t}$ is an equilibrium measure for $t\Phi_{\mathcal{A}}$. Then uniqueness of equilibrium measure implies $\mu_{t_{k}} \rightarrow \mu_{t}$. Moreover, we have equality in the above, which gives the claim. Furthermore, it shows the continuity of the Lyapunov exponent of equilibrium measures.
\end{proof}

Assume that $(\mu_{t})$ is a sequence of equilibrium measures for a subadditive potential $t\log \Phi$, where $t>0$. The author \cite{Moh20} investigated the behavior of the equilibrium measure $(\mu_{t})$ as $t\rightarrow \infty$. In the thermodynamic interpretation of the parameter $t$, it is the \textit{inverse temperature}. The limits $t\rightarrow \infty$ is called \textit{zero temperature limits}, and the  accumulation points of the measure $(\mu_{t})$ as $t\rightarrow \infty$ are called \textit{ground states}. 

Zero temperature limit is also related to ergodic optimization, as for $t\rightarrow \infty$ any accumulation points of equilibrium measure $\mu_{t}$ is a maximizing measure $\Phi.$

\begin{thm}[{{\cite[Theorem 1.1]{Moh20}}}]\label{cont1en}
Let $(X,T)$ be a topological dynamical system for which the entropy map $\mu \mapsto h_{\mu}(T)$ is upper semi-continuous and topological entropy $h_{top}(T)<\infty.$ Assume that $\Phi=\{\log \phi_{n}\}_{n=1}^{\infty}$ is a subadditive potential on the compact metric $X$. Then a family of equilibrium measures $(\mu_{t})$ for potentials $t\Phi$, where $t>0$, has a weak*  accumulation point as $t \rightarrow \infty.$ Any such accumulation point  is a Lyapunov maximizing measure for $\Phi$. Furthermore,
\begin{itemize}
\item[(i)] $\chi(\Phi, \mu)=\lim_{i\rightarrow \infty} \chi(\Phi, \mu_{t_{i}}),$
\item[(ii)]$h_{\mu}(T)=\lim_{i \rightarrow \infty} h_{\mu_{t_{i}}}(T).$
\end{itemize}
Moreover, above result works for higher dimensional case.
\end{thm}

We use the above theorem to prove the following theorem.

\begin{thm}\label{1dim}
 Suppose that $\mathcal{\mathcal{A}} \in \mathcal{W}$. If $\alpha_{t}= P_{\Phi_{\mathcal{A}}}^{'}(t)$ for $t>0.$ Then
\[ h_{top}(E(\alpha_{t})) \rightarrow h_{top}(E(\beta(\mathcal{A})) \hspace{0,2cm}\textrm{when}\hspace{0,2cm} t \rightarrow \infty.\]
\end{thm}
\begin{proof}
Since $\mathcal{\mathcal{A}}\in \mathcal{W}$, Theorem \ref{uniq11} implies that for $t>0$, there is a unique equilibrium state $\mu_{t}$ for the subadditive potential $t \Phi_{\mathcal{A}}$ such that \[\chi(\mu_{t}, \mathcal{A})=\alpha_{t}=P_{\Phi_{\mathcal{A}}}^{'}(t).\] 

By Theorem \ref{derivative}, \[h_{top}(E(\alpha_{t}))=h_{\mu_{t}}(T).\]

 We know that \[h_{top}(E(\beta(\mathcal{A}))=\sup\{h_{\mu}(T), \hspace{0,2cm} \mu \in \mathcal{M}(\Sigma, T), \hspace{0,2cm} \chi(\mu, \mathcal{A})=\beta(\mathcal{A})\}\] by Theorem \ref{level set1}. So, we only need to show that
\[ h_{\mu_{t}}(T) \rightarrow \sup\{h_{\mu}(T), \hspace{0,2cm} \mu \in \mathcal{M}(\Sigma, T), \hspace{0,2cm} \chi(\mu, \mathcal{A})=\beta(\mathcal{A})\}.\]
That follows from Theorem \ref{cont1en}.
\end{proof}
\begin{proof}[Proof of Theorem C]It follows from Theorem \ref{1dim} and \ref{convergence1}.
\end{proof}

\begin{thm}\label{thmC}
Suppose $\mathcal{A}_{l}, \mathcal{A}\in \mathcal{W}$ with $\mathcal{A}_{l} \rightarrow A$, and $\vec{t_{l}},\vec{t}\in \R_{+}^{k}$ such that $t_{l}\rightarrow t.$
Assume that $\vec{\alpha_{t_{l}}}= \nabla P_{\vec{\Phi}_{\mathcal{A}_{l}}}(\vec{t_{l}})$ and $\vec{\alpha_{t}}=\nabla P_{\vec{\Phi}_{\mathcal{A}}}(\vec{t})$. Then,
\[\lim_{l\rightarrow \infty} h_{top}(E(\vec{\alpha_{t_{l}}})) =h_{top}(E(\vec{\alpha_{t}})).\]
Moreover,
\[ h_{top}(E(\vec{\alpha_{t}})) \rightarrow h_{top}(E(\vec{\beta}(\vec{\Phi}_{\mathcal{A}})) \hspace{0,2cm}\textrm{when}\hspace{0,2cm} t \rightarrow \infty.\]
\end{thm}
\begin{proof}
The proof is similar to Theorems \ref{1dim} and \ref{cont1en} and is omitted.
\end{proof}

\section{The proof of Theorem D}\label{5}
In this section, we are going to prove the continuity of the lower joint spectral radius for general cocycles under certain assumptions. This kind of result is known by Bochi and Morris \cite{BM} under 1-domination assumption for locally constant cocycles.  Breuillard and Sert \cite{BS} extended their result to the joint spectrum of locally constant cocycles. Moreover, they gave a counterexample \cite[Example 4.13]{BS} that shows that we have discontinuity the lower joint spectral for typical cocycles. Even though, we have a lot of results for the upper spectral radius, we have  few result about the lower spectral radius, which shows that working on the latter case is much harder than the former case.

Assume that $T:X\rightarrow X$ is a diffeomorphism on a compact invariant set $X$. Let $V\oplus W$ be a splitting of the tangent bundle over $X$ that is invariant by the tangent map $DT$. In this case, if vectors in $V$ are uniformly contracted by $DT$ and vectors in $W$ are uniformly expanded, then the splitting is called \textit{hyperbolic}. The more general notion is the \textit{dominated splitting}, if at each point all vectors in $V$ are more contracted than all vectors in $W$. Domination could also be called uniform projective hyperbolicity. Indeed, domination is  equivalent to $V$ being hyperbolic repeller and $W$ being hyperbolic attractor in the projective bundle.

 In the linear cocycles, we are interested in bundles of the form $X\times \R^{k}, $ where the linear cocycles are generated by $(T, \mathcal{A})$. Bochi and Gourmelon \cite{BGO} showed that a cocycle admits a dominated splitting $V\oplus W$ with $\dim V=k$ if and only if when $n\rightarrow \infty$, the ratio between the $k-th$ and $(k+ 1)-th$ singular values of the matrices of the $n-th$ iterate increase uniformly exponentially. In fact, they extended the Yoccoz's result \cite{Y} that was proved for 2-dimensional vector bundles.

\begin{defn}
We say that $\mathcal{A}$ is \textit{i-dominated} if there exist constants $C >1$, $0<\tau<1$ such that
\[\frac{\sigma_{i+1}(\mathcal{A}^{n}(x))}{\sigma_{i}(\mathcal{A}^{n}(x))}\leq C \tau^n, \hspace{0.2cm} \forall n\in \N, x\in X.\]%Let $\Upsilon$ be a compact subset of $GL(k, \R)$. We say that $\Upsilon$ is \textit{k-dominated}, or that $k$ is an index of domination for $\Upsilon$, if and only if there are constants $C>0$ and $0<\tau<1$ such that for any finite sequence $A_{1},...,A_{N}$ in $\Upsilon$ we have 

%For any dynamics $T: X \rightarrow X$ and any map $\mathcal{A}:X \rightarrow GL(k, \R)$, we say that the cocycle $(T, \mathcal{A})$ (or $\mathcal{A}$) is $k$-dominated if whose image is
%contained in $\Upsilon$.
\end{defn}

 According to the multilinear algebra properties, $\mathcal{A}$ is $i-$dominated if and only if $\mathcal{A}^{\wedge i}$ is $1-$dominated.

%\begin{thm}
%Suppose $\{\nu_{n}\}_{n=1}^{\infty}$ is a sequence in $\mathcal{M}(X)$ and $\mathcal{A}$ is 1-domination.  We form the new sequence $\{\mu_{n}\}_{n=1}^{\infty}$ by $\mu_{n}=\frac{1}{n}\sum_{i=0}^{n-1}\nu_{n}oT^{i}$.  Assume that $\mu_{n_{i}}$ converges to $\mu$ in $\mathcal{M}(X)$ for some subsequence $\{n_i\}$ of natural numbers.  Then $\mu\in \mathcal{M}(X,T)$ and 
%\begin{equation}\label{add1}
%\lim_{i \rightarrow \infty} \frac{1}{n_{i}}\int \log \|\mathcal{A}^{n_{i}}(x)\|d\nu_{i}(x)=\chi (\mu, \mathcal{A}).
%\end{equation} 
%\end{thm}
%\begin{proof}
% 1-domination implies almost additivity of $\|\mathcal{A}\|$ (see \cite{BM}). Then assertion follows from [\cite[CFH],...].
%\end{proof}

Let $(X,T)$ be a TDS. We say that $\mathcal{A}:X\rightarrow GL(k, \R)$ is \textit{almost multiplicative} if there is a constant $C>0$ such that
\[||\mathcal{A}^{m+n}(x)|| \geq C ||\mathcal{A}^m(x)|| ||\mathcal{A}^n(T^m(x))|| \hspace{0.2cm}\forall x\in X, m.n\in \N.\]

We note that since clearly $||\mathcal{A}^{m+n}(x)|| \leq  ||\mathcal{A}^m(x)|| ||\mathcal{A}^n(T^m(x))|| \hspace{0.1cm}$ for all $x\in X, m,n\in \N$, the condition of almost multiplicativity of $\mathcal{A}$ is equivalent to the statement that $\Phi_{\mathcal{A}}$ is almost additive.\\

\subsection{Almost additive thermodynamic formalism}
In this subsection, we state a theorem that shows  we have the Bowen's result for almost additive sequences.

Let $T:\Sigma \to \Sigma$ be a topologically mixing subshift of finite type. We say that a subadditive sequences $\Phi:=\{\log \phi_{n}\}_{n=1}^{\infty}$ over  $(\Sigma, T)$ has \textit{bounded distortion:}  there exists $C\geq 1$ such that for any $n\in \N$ and $I \in \mathcal{L}(n)$, we have
\[ C^{-1} \leq \frac{\phi_{n}(x)}{\phi_{n}(y)} \leq C\]
for any $x, y \in [I].$

\begin{thm}[{{\cite[Theorem 10.1.9]{Bar}}}]\label{almost-unique}
Let $\Phi=\{\log \phi_{n}\}_{n=1}^{\infty}$ be an almost additive sequence over a topologically mixing subshift of finite type $(\Sigma, T)$. Assume that $\Phi$ has bounded distortion. Then :
\begin{itemize}
\item[1.]There is a unique equilibrium measure for $\Phi$,
\item[2.]there is a unique invariant Gibbs measure for $\Phi$,
\item [3.]the two measures coincide and are ergodic.
\end{itemize}
\end{thm}

\begin{thm}\label{derivative:aa}
Let $t\in \R^{k}$ and  let $\vec {\Phi}$ be an almost additive potential. Then, Proposition \ref{eq11} holds for $t\in \R^{k}$.
\end{thm}
\begin{proof}
It follows from \cite[Theorem 3.3]{FH}.
\end{proof}

Given an almost additive potential $\Phi=\{\log \phi_{n}\}_{n=1}^{\infty}$. Feng and Huang \cite[ Lemma  A.4.]{FH} proved the following lemma:
\begin{lem}\label{AAP}
Let $\mu\in \mathcal{M}(X, T)$. Then the map $\mu \mapsto \chi(\mu, \Phi)$ is continuous on $\mathcal{M}(X, T)$.
\end{lem}

\subsection{A cocycle is almost multiplicativite under a cone condition}

Let $\mathcal{A}:X\to GL(k, \R)$ be a matrix cocycle over a TDS $(X,T)$. We say that a family of convex cones $\mathcal{C}:=(C_x)_{x\in X}$ is a $\ell$-dimensional invariant cone field on $X$ if
\[\mathcal{A}(x)C_{x} \subset C_{T(x)}^{o} \hspace{0.2cm}\forall x\in X,\]
where $\ell$ is the dimension $\mathcal{C}.$

 Let $V$ be a vector space over the reals.
\begin{defn}
Fix a convex cone $C \subset V$. Given $v, w \in  C$, let 
\begin{equation}\label{Hil1}
\alpha(v,w)=\sup\{\lambda >0 | w-\lambda v \in C  \}, \hspace{0.2cm}\beta(v,w)=\inf\{\mu >0 | \mu v-w \in C\},
\end{equation}
with $\alpha = 0$ and/or $\beta = \infty$ if the corresponding set is empty. The cone
distance between $v$ and $w$ is
\begin{equation}\label{Hil2}
d_{C}(v, w)=\log \frac{\beta(v, w)}{\alpha(v, w)}.
\end{equation} 
The distance $d_{c}$ is called Hilbert projective (pseudo) metric.
\end{defn}
\begin{thm}\label{contraction}
 Let $C_{1} \subset V_{1}$ and $C_{2} \subset V_{2}$ be convex cones, and $L:V_{1}\rightarrow V_{2}$ be a linear map such that $L(C_{1}) \subset C_{2}^{o}$. Assume that $\triangle:=\sup_{\hat{v}, \hat{\psi}\in L(C_{1})} d_{C_{2}}(\hat{v}, \hat{w}).$ Then for all $v,w\in C_{1}$, we have 
\[d_{C_{2}}(Lv, Lw) \leq \tanh(\frac{\triangle}{4})d_{C_{1}}(v, w),\]
where we use the convention that $\tanh\infty = 1.$
\end{thm}
\begin{proof}
See \cite{NL}.
\end{proof}
We are also going to use the following lemma in Proposition \ref{prop:add}.

\begin{lem}\label{bounded_hilbert}Let $V$ be a finite dimensional vector space. Suppose that $C_{1}$ and $C_{2}$ are two convex cones in $V$ such that $C_1\subset C_2^o$ and $d_{C_{2}}$ is the Hilbert metric on $C_2$. Then $C_1$ is bounded in metric $d_{C_2}$. 
\end{lem}
\begin{proof}
Let us denote $d$ as the usual distance on the projective space. Since $C_{1} \subset C_{2}^{o}$, $d(C_{1}, \partial C_{2})>0$. Hence, for every $v,w \in C_{1}$  the distances $d(A,v), d(B,w)$ are uniformly bounded from below by $c_1=d(C_1, \partial C_2)$, where $A,B$ are the intersections of the line $\overline{vw}$ with $\partial C_1$. On the other hand, $d(v,w)$ is uniformly bounded from above by $c_2=\diam_d(C_1)$. Thus, $d_{C_2}(v,w) \leq \log ( (c_1+c_2)/c_1)$.

\end{proof}

\begin{prop} \label{prop:add}
Let $X$ be a compact metric space, and let $\mathcal{A}: X\rightarrow GL(k, \R)$ be a matrix cocycle over a TDS $(X,T)$. Let $(C_r)_{r\in X}$ be a 1-dimensional invariant cone field on $X$. Then, there exists $\kappa>0$ such that for every $m,n>0$ and for every $x\in X$ we have

\[
||\mathcal{A}^{m+n}(x)|| \geq \kappa ||\mathcal{A}^m(x)|| \cdot ||\mathcal{A}^n(T^m(x))||.
\]
\end{prop}

\begin{proof}
Let us start from the notation. Denote by $\pi$ the natural projection from $\R^{k}$ to the projective space $\mathbb{P}\R^{k}$ and by $d$ the natural metric on $\mathbb{P}\R^{k}$. For a family of convex cones $(C_r)_{r\in X}$, all contained in the interior $C^o$ of another convex cone $C$, we define their convex hull as

\[
\conv(C_r) = \{v\in C; \pi(v) = \pi(\sum_i \alpha_i v_i) \hspace{0.1cm}\textrm{for some}\hspace{0.1cm} \alpha_i\geq 0, \sum_i \alpha_i=1, v_i\in C_{r_i}\}
\]
The Hausdorff distance in metric $d$ between $C$ and $\conv(C_r)$ equals the supremum of Hausdorff distances between $C$ and $C_r$ (to be absolutely precise, the Hausdorff distance is defined for compact sets and the metric $d$ is defined on the projective space, so we mean here the Hausdorff distance between $\overline{\pi(C)}$ and $\overline{\pi(\conv(C_r))}$).This supremum is positive, hence $\conv(C_r)\subset C^o$.

For every $x\in X$ the set $T^{-1}(x)$ is compact. Thus, we can define

\[
D_x = \conv(\{\mathcal{A}(y)C_y; y\in T^{-1}(x)\})
\]
for $x\in T(X)$ and, by compactness, we have $D_x\subset C_x^o$. 
 We choose $D_x$ as any convex cone contained in $C_x^o$ for $x\in X\setminus T(X)$, we only demand that $x\to D_x$ is a continuous map (this can be done because $X$ is compact, hence $X\setminus T(X)$ is open in $X$).
One can check that, as $D_x\subset C_x^o$, we have

\[
\mathcal{A}(x) D_x \subset (\mathcal{A}(x) C_x)^o \subset D_{T(x)}^o.
\]
Hence, $(D_x)$ is another invariant cone field, strictly contained in $(C_x)$.

Let for each $x\in X$ $d_x$ be the Hilbert metric in $C_x$. Let $d$ be the usual metric on $\mathbb{P}\R^{k}$. We have the following properties:

\begin{itemize}
\item Each $D_x$ is bounded in $d_x$ by Lemma \ref{bounded_hilbert}. By compactness of $X$, there exists $K_1>0$ such that $\diam_{d_x}(D_x)<K_1$ for all $x\in X$.
\item In each $D_x$ the metric $d_x$ is equivalent to $d$. By compactness of $X$, there exists $K_2>1$ such that for every $x\in X$ for every $v,w\in D_x$ we have $K_2^{-1} d_x(v,w) \leq d(v,w) \leq K_2 d_x(v,w)$.
\item Each $\mathcal{A}(x): D_x \to D_{T(x)}$ is a contraction by Theorem \ref{contraction}. By compactness of $X$, there exists $\lambda<1$ such that for every $x\in X$ for every $v,w\in D_x$ we have $d_{T(x)}(\mathcal{A}(x)v, \mathcal{A}(x)w) \leq \lambda d_x(v,w)$.
\item For $v\in C_x$ denote \begin{equation}\label{fun}
\gamma_x(v):= \log (\|\mathcal{A}(x)v\|/\|v\|) \in \R.
\end{equation} The map $v\to\gamma_x(v)$ is Lipschitz (in metric $d$) on $D_x$. By compactness of $X$, there exists $K_3>0$ such that for every $x\in X$ the map $\gamma_x$ is $K_3$-Lipschitz (in metric $d$) on $D_x$.
\item For every $x\in X$ the convex cone $D_x$ contains (for some $v_x\in D_x\cap \mathbb{P}\R^{k}$ and $r_x>0$) a ball $B(v_x,r_x)=\{w\in \mathbb{P}\R^{k}; d(v_x,w)<r_x\}$. By compactness of $X$, there exists $r>0$ such that for every $x\in X$ we have $D_x \supset B(v_x,r)$ for some $v_x\in D_x\cap \mathbb{P}\R^{k} $.
\end{itemize}

Choose some $x\in X$ and $v,w\in D_x$. Fix $m>0$. Denote

\[
\gamma_x^m(v) = \log \frac {\|\mathcal{A}^m(x)v\|} {\|v\|} = \sum_{i=0}^{m-1} \gamma_{T^i(x)}(\mathcal{A}^i(x)v).
\]
Note three obvious properties of this function:
\begin{itemize}
\item $\gamma_x$ is a projective function, that is $\gamma_x (v)=\gamma_x (\alpha v)$ for $\alpha>0$. Thus, we can define $\gamma_x$ on the projective space $\mathbb{P}\R^{k}$. The same holds for $\gamma_x^m$.
\item $\gamma_x^m(v) \leq \log ||\mathcal{A}^m(x)||$,
\item $\gamma_x^{m+n}(v) = \gamma_x^m(v) + \gamma_{T^m(x)}^n(\mathcal{A}^m(x) v)$.
\end{itemize}

We have
\begin{align*}
d(\mathcal{A}^i(x)v, \mathcal{A}^i(x)w) &\leq K_2 d_{T^i(x)} (\mathcal{A}^i(x)v, \mathcal{A}^i(x)w)\\ &\leq K_2 \lambda^i d_x(v,w) \leq K_2 \lambda^i K_1.
\end{align*}
Hence (see \eqref{fun}),
\[
|\gamma_{x}^m(v)-\gamma_{x}^m(w)| \leq K_3 \sum_{i=0}^{m-1} d(\mathcal{A}^i(x)v, \mathcal{A}^i(x)w)  \leq K_4:= K_1 K_2 K_3 \frac 1 {1-\lambda}
\]
for every $v,w\in D_x$.

To finish the proof we need the following lemma.

\begin{lem} \label{lem:supp}
Let $A\in GL(k, \R)$ and let $K,r>0$. Assume $|\gamma(v)-\gamma(w)|<K$ for some $v\in \mathbb{P}\R^{k}$ and all $w\in B(v,r)$, where $\gamma(v)=\log \|Av\|$. Then there exists a constant $\rho=\rho(K,r)$, depending on $K$ and $r$ but not on $A$, such that $\gamma(v) \geq \log ||A|| - \rho(K,r)$.
\end{lem}

Before proving Lemma \ref{lem:supp} let us observe that it indeed implies the assertion of Proposition \ref{prop:add}. As $D_x$ contains some ball $B(v,r)$ with $v\in D_x\cap \mathbb{P}\R^{k}$, we can apply the lemma to the matrix $\mathcal{A}^m(x)$, obtaining $\log ||\mathcal{A}^m(x)|| \leq \rho(K_4,r) + \gamma_x^m(v)$. Hence, for every $w\in D_x$ we have

\[
\log ||\mathcal{A}^m(x)|| \leq \rho(K_4,r) + K_4+ \gamma_x^m(w).
\]
Similarly, $D_{T^m(x)}$ contains a ball of size $r$, hence for every $u\in D_{T^m(x)}$ we have
\[
\log ||\mathcal{A}^n(T^m(x))|| \leq \rho(K_4,r) + K_4+ \gamma_{T^m(x)}^n(u).
\]
Thus, choosing $u=\mathcal{A}^m(x) w$ we get

\begin{align*}
\log ||\mathcal{A}^m(x)|| + \log ||\mathcal{A}^n(T^m(x))|| &\leq 2 \rho(K,r) + 2K_4 + \gamma_x^{m+n}(w)\\ & \leq 2 \rho(K,r) + 2K_4 + \log ||\mathcal{A}^{m+n}(x)||
\end{align*}
which is our assertion.

Now, let us come back and prove Lemma \ref{lem:supp}.

\begin{proof}
We start by a decomposition $A=O_1 D O_2$, where $O_1, O_2$ are orthogonal matrices and $D$ is a diagonal matrix with elements $\pm(\sigma_i(A))$ (the singular values of $A$). It is enough to prove the assertion for the matrix $D$.

So, let $D$ be a diagonal matrix. Let $e$ be the eigenvector corresponding to the maximal eigenvalue: $|De|=||D||$. Even when $v.e=0$, the ball $B(v,r)$ still must contain a vector $w$ such that $|w.e|\geq 1/2 \cdot\sin r$. We have $w=|w.e|e + (1-(w.e)^2)^{1/2}e'$, where $e.e'=0$. Hence,

\[
\gamma(w) = \log |Dw| \geq \log (|w.e| \cdot |De|) = \log |w.e| + \log ||D|| \geq \log (\frac 12 \sin r) + \log ||D||.
\]
Thus, for every $u\in B(v,r)$ we have

\[
\gamma(u) \geq \gamma(w) - K \geq \log ||D|| + \log (\frac 12 \sin r)  - K.
\]
We are done.
\end{proof}
\end{proof}

Given $l$-dimensional subspace $E(x)\subset T_{x} \R^{k}$, we define the convex cone
\[C_{x}=\left\{(u, v) \in E(x) \oplus E(x)^{\perp}:\|v\| \leq \|u\|\right\}.\]
We say that a differentiable map $f: \mathbb{R}^{k} \rightarrow \mathbb{R}^{k}$ satisfies a $l-$\textit{cone condition} on a set $\Lambda \subset \mathbb{R}^{k}$ if for each $x \in \Lambda$ a $l$-dimensional subspace $E(x) \subset T_{x} \mathbb{R}^{k}$ varying continuously with $x$ such that
\[
\left(d_{x} f\right) C_{x} \subset C_{f(x)}^{o}.
\]
\begin{cor}
Let $M$ be a compact manifold in dimension $d$. Let $T\in \Diff^1(M)$. Let $\mathcal{A}(x)=DT(x)\in GL(k,\R)$ satisfying a $1-$cone condition. Then, there exists $\kappa>0$ such that for every $m,n>0$ and for every $x\in M$ we have

\[
||\mathcal{A}^{m+n}(x)|| \geq \kappa ||\mathcal{A}^m(x)|| \cdot ||\mathcal{A}^n(T^m(x))||.
\]
\end{cor}
\begin{proof}
It follows from the proof of Proposition \ref{prop:add}.
\end{proof}

One needs to be careful that quasi-multiplicativity is not equivalent of almost additivity. For instance, let  $T:\{0,1\}^{\Z}\rightarrow \{0,1\}^{\Z}$ be a shift map. We define a linear cocycle $(T, \mathcal{A})$ such that
\begin{equation}\label{example}
A_{0}=\left[\begin{array}{ccccc} 2 & 0\\ 0 & \frac{1}{2}
\end{array} \right],  \hspace{0.2cm} A_{1}=R_{\theta},
\end{equation}
where $\theta$ is an irrational angle. It is easy to see that the locally constant cocycle $(T, \mathcal{A})$ is strongly irreducible. Feng \cite[Proposition 2.8]{F} showed that the irreducible matrix cocycles are quasi-multiplicative.

\subsection{The proof of Theorem D}

We say that $\hat{\rho}(\mathcal{A})$ is  the \textit{upper joint spectral radius} of $\mathcal{A}: X \rightarrow GL(k, \R)$ if
\[ \log \hat{\rho}(\mathcal{A})=\beta(\mathcal{A}).\]

 Similarly, we say that $\check{\rho}(\mathcal{A})$ is the \textit{lower joint spectral radius} of $\mathcal{A}: X \rightarrow GL(k, \R)$ if 
\[ \log\check{\rho}(\mathcal{A})=\alpha(\mathcal{A}).\]

Assume that $f:X\rightarrow X$ is a convex continuous function on the compact metric space $X$. We have $\overline{\partial f(\R)}=\partial f(\R) \cup  \{f^{'}(\infty)\}\cup  \{f^{'}(-\infty)\}$,  where $\partial f(\R) $ is defined as in (\ref{dense222}).

\begin{thm}\label{app}
Let $(X, T)$ be a TDS such that the entropy
map $\mu \mapsto h_{\mu}(T)$ is upper semi-continuous and $h_{top}(T)<\infty$. Suppose that $\mathcal{A}:X\rightarrow GL(k, \R)$ is a matrix cocycle over the TDS $(X, T)$. Assume that $(C_x)_{x\in X}$ is a 1-dimensional invariant cone field on $X$. Then $\alpha(\mathcal{A})$ can be approximated by 
the Lyapunov exponents of the equilibrium measures for the almost additive potential $t\Phi_{\mathcal{A}}$, where $t\in \R$. Moreover, a Lyapunov minimizing measure for $\mathcal{A}$ exists.
\end{thm}
\begin{proof}
We know that $\mathcal{A}$ is almost multiplicative by Proposition \ref{prop:add}. 
Let $\alpha:=\alpha(\mathcal{A})=P_{\Phi_{\mathcal{A}}}^{'}(-\infty)$ (see Theorem \ref{convex11}). 

By the convexity of $P_{\Phi_{\mathcal{A}}}()$, there exists a sequence $(t_{j} )$ such that $ P_{\Phi_{\mathcal{A}}}^{'}(t_j )=:
\alpha_{j}$ exists for every $j\in \N$ and $\alpha_{j}\rightarrow \alpha$ as $j \rightarrow \infty$. There exists $\mu_{j} \in \Eq(\Phi_{\mathcal{A}}, t_{j} )$
such that $\chi(\mu_{j}, \Phi)=\alpha_{j}$ for all $j$, by Theorem \ref{derivative:aa}. Extract a subsequence if necessary so that $\mu_{j} \to \mu$ (weak* topology) \footnote{$\Eq(
\Phi_{\mathcal{A}}, t)$ is compact in weak* topology.} as $j\rightarrow \infty$. By Lemma \ref{AAP}, we have
\[\chi(\mu_{j}, \mathcal{A}) \rightarrow \chi(\mu, \mathcal{A})=\alpha.\]

Furthermore, our proof shows that a Lyapunov minimizing measure exists.

\end{proof}
Now, we can show the continuity of the minimal Lyapunov exponent.

 \begin{proof}[Proof of Theorem D]
According to Theorem \ref{app}, $\alpha(\mathcal{A})$ can be approximated by 
 Lyapunov exponents of equilibrium measures for the almost additive potential $t\Phi_{\mathcal{A}}$, where $t\in \R$. Thus, it is enough to show that \begin{equation}\label{r88}
 \chi(\mu_{n}, \mathcal{A}_{n})\rightarrow \chi(\mu, \mathcal{A}),
\end{equation}  
 where $\mu, \mu_{n}$ are the equilibrium measures.
 
By Proposition \ref{prop:add}, $\mathcal{A}$ is almost multiplicative and $\Phi_{\mathcal{A}}$ has bounded distortion. Hence, there exist a unique equilibrium measure for the almost additive potential $t\Phi_{\mathcal{A}}$ by Theorem \ref{almost-unique}. Therefore, (\ref{r88}) follows from  Lemma \ref{AAP}.
 \end{proof}

 Domination can be characterized in terms of the existence of invariant cone fields for derivative cocycles (\cite[Theorem 2.6]{CP}). This fact shows that 1-domination implies that $\mathcal{A}$ is almost multiplicative. Hence, one can prove Theorem D for fiber bunched cocycles \footnote{See \cite[Lemma 3.10]{P}.} under 1-domination assumption.

  It is possible to obtain the generalization of Theorem \ref{app} to the joint spectrum of all Lyapunov exponents. One can also obtain the continuity of the lower joint spectral radius for all Lyapunov exponents.

\bibliography{ref}

\end{document}